\documentclass{amsart}

\usepackage{graphicx}
\usepackage{amssymb}
\usepackage{amsfonts}
\usepackage{amsthm}
\usepackage{mathtools}
\usepackage{dsfont}
\usepackage{subfig}
\usepackage{etex}
\usepackage[all,cmtip]{xy}
\usepackage{url}
\usepackage[usenames,dvipsnames]{xcolor}\usepackage[bookmarks=true,colorlinks=true, linkcolor=RoyalBlue, citecolor=magenta]{hyperref}

\usepackage{tikz}
\usetikzlibrary{decorations.markings,positioning}
\usetikzlibrary{shapes.geometric}
\usepgflibrary{decorations.shapes}
\usetikzlibrary{decorations.shapes, shapes, fit, arrows, positioning, trees, mindmap, calc}
\usepackage{array}

\tikzset{
  myblock/.style={
    draw,text width=0.75cm,minimum height=2cm,align=center,
    append after command={node[fill,anchor=north west,minimum size=0cm] at (\tikzlastnode.north west) {}}
  },
  mysquare/.style={
    draw,minimum size=.8cm,align=center,
    append after command={node[fill,anchor=north west,minimum size=0cm] at (\tikzlastnode.north west) {}}
  },
  twoarrows/.style n args={4}{
    decoration={
      markings,
      mark=at position #1 with {\arrow{>}\node[above] {#3};}, 
      mark=at position #2 with {\arrow{>}\node[above] {#4};} 
    },
  postaction=decorate  
  },
  onearrow/.style 2 args={
    decoration={
      markings,
      mark=at position #1 with {\arrow{>}\node[above] {#2};}, 
    },
  postaction=decorate  
  },
   fold/.style={
            draw=black,
            isosceles triangle,
            fill=white,
            minimum height=1cm,
            minimum width =0.05cm,
        }
}

\def\blackbox#1#2#3#4#5{
  \pgfgetlastxy{\llx}{\lly}
  \path #1;
  \pgfgetlastxy{\w}{\h}
  \pgfmathsetlengthmacro{\urx}{\llx+\w}
  \pgfmathsetlengthmacro{\ury}{\lly+\h}
  \draw (\llx,\lly) rectangle (\urx,\ury);
  \pgfmathsetlengthmacro{\xave}{(\llx+\urx)/2}
  \pgfmathsetlengthmacro{\yave}{\ury-8}
  \node at (\xave,\yave) {#4};
  \pgfmathsetlengthmacro{\ydiff}{\ury-\lly}
  \pgfmathsetlengthmacro{\lstep}{\ydiff/(#2+1)}
  \pgfmathsetlengthmacro{\rstep}{\ydiff/(#3+1)}
  \ifnum #2=0{}\else{ 
   \foreach \l in {1,...,#2}{
    \draw [->] ($(\llx,\lly)+(-#5/2,0)+\l*(0,\lstep)$) -- ($(\llx,\lly)+(#5/2,0)+\l*(0,\lstep)$);}}\fi
  \ifnum #3=0{}\else{
   \foreach \r in {1,...,#3}{
    \draw [->] ($(\urx,\ury)+(-#5/2,0)-\r*(0,\rstep)$) -- ($(\urx,\ury)+(#5/2,0)-\r*(0,\rstep)$);}}\fi
}

\def\blackboxinners#1#2#3#4#5{
  \pgfgetlastxy{\llx}{\lly}
  \path #1;
  \pgfgetlastxy{\w}{\h}
  \pgfmathsetlengthmacro{\urx}{\llx+\w}
  \pgfmathsetlengthmacro{\ury}{\lly+\h}
  \draw (\llx,\lly) rectangle (\urx,\ury);
  \pgfmathsetlengthmacro{\xave}{(\llx+\urx)/2}
  \pgfmathsetlengthmacro{\yave}{\ury-8}
  \node at (\xave,\yave) {#4};
  \pgfmathsetlengthmacro{\ydiff}{\ury-\lly}
  \pgfmathsetlengthmacro{\lstep}{\ydiff/(#2+1)}
  \pgfmathsetlengthmacro{\rstep}{\ydiff/(#3+1)}
  \ifnum #2=0{}\else{ 
   \foreach \l in {1,...,#2}{
    \pgfmathsetlengthmacro{\newx}{\llx+#5*28.45274/2}
    \pgfmathsetlengthmacro{\newy}{\lly+\l*\lstep}
    \node at ($(\newx,\newy)+(-1.5,\l*12-\l*\lstep)$) {\tiny$(\pgfmathparse{\newx/28.45274}\pgfmathresult cm,\pgfmathparse{\newy/28.45274}\pgfmathresult cm)$};
    \draw [->] ($(\llx,\lly)+(-#5/2,0)+\l*(0,\lstep)$) -- ($(\llx,\lly)+(#5/2,0)+\l*(0,\lstep)$);}}\fi
  \ifnum #3=0{}\else{
   \foreach \r in {1,...,#3}{
    \pgfmathsetlengthmacro{\newx}{\urx-#5*28.45274/2}
    \pgfmathsetlengthmacro{\newy}{\ury-\r*\rstep}
    \node at ($(\newx,\newy)+(1.5,-\r*12+\r*\rstep)$) {\tiny $(\pgfmathparse{\newx/28.45274}\pgfmathresult cm,\pgfmathparse{\newy/28.45274}\pgfmathresult cm)$};
    \draw [->] ($(\urx,\ury)+(-#5/2,0)-\r*(0,\rstep)$) -- ($(\urx,\ury)+(#5/2,0)-\r*(0,\rstep)$);}}\fi
}

\def\blackboxouters#1#2#3#4#5{
  \pgfgetlastxy{\llx}{\lly}
  \path #1;
  \pgfgetlastxy{\w}{\h}
  \pgfmathsetlengthmacro{\urx}{\llx+\w}
  \pgfmathsetlengthmacro{\ury}{\lly+\h}
  \draw (\llx,\lly) rectangle (\urx,\ury);
  \pgfmathsetlengthmacro{\xave}{(\llx+\urx)/2}
  \pgfmathsetlengthmacro{\yave}{\ury-8}
  \node at (\xave,\yave) {#4};
  \pgfmathsetlengthmacro{\ydiff}{\ury-\lly}
  \pgfmathsetlengthmacro{\lstep}{\ydiff/(#2+1)}
  \pgfmathsetlengthmacro{\rstep}{\ydiff/(#3+1)}
  \ifnum #2=0{}\else{ 
   \foreach \l in {1,...,#2}{
    \pgfmathsetlengthmacro{\newx}{\llx-#5*28.45274/2}
    \pgfmathsetlengthmacro{\newy}{\lly+\l*\lstep}
    \node at ($(\newx,\newy)+(-1.5,\l*12-\l*\lstep)$) {\tiny$(\pgfmathparse{\newx/28.45274}\pgfmathresult cm,\pgfmathparse{\newy/28.45274}\pgfmathresult cm)$};
    \draw [->] ($(\llx,\lly)+(-#5/2,0)+\l*(0,\lstep)$) -- ($(\llx,\lly)+(#5/2,0)+\l*(0,\lstep)$);}}\fi
  \ifnum #3=0{}\else{
   \foreach \r in {1,...,#3}{
    \pgfmathsetlengthmacro{\newx}{\urx+#5*28.45274/2}
    \pgfmathsetlengthmacro{\newy}{\ury-\r*\rstep}
    \node at ($(\newx,\newy)+(1.5,-\r*12+\r*\rstep)$) {\tiny $(\pgfmathparse{\newx/28.45274}\pgfmathresult cm,\pgfmathparse{\newy/28.45274}\pgfmathresult cm)$};
    \draw [->] ($(\urx,\ury)+(-#5/2,0)-\r*(0,\rstep)$) -- ($(\urx,\ury)+(#5/2,0)-\r*(0,\rstep)$);}}\fi
}

\def\dashbox#1#2#3#4#5{
  \pgfgetlastxy{\llx}{\lly}
  \path #1;
  \pgfgetlastxy{\w}{\h}
  \pgfmathsetlengthmacro{\urx}{\llx+\w}
  \pgfmathsetlengthmacro{\ury}{\lly+\h}
  \draw [dashed] (\llx,\lly) rectangle (\urx,\ury);
  \pgfmathsetlengthmacro{\xave}{(\llx+\urx)/2}
  \pgfmathsetlengthmacro{\yave}{\ury-8}
  \node at (\xave,\yave) {#4};
  \pgfmathsetlengthmacro{\ydiff}{\ury-\lly}
  \pgfmathsetlengthmacro{\lstep}{\ydiff/(#2+1)}
  \pgfmathsetlengthmacro{\rstep}{\ydiff/(#3+1)}
  \ifnum #2=0{}\else{ 
   \foreach \l in {1,...,#2}{
    \draw [->] ($(\llx,\lly)+(-#5/2,0)+\l*(0,\lstep)$) -- ($(\llx,\lly)+(#5/2,0)+\l*(0,\lstep)$);}}\fi
  \ifnum #3=0{}\else{
   \foreach \r in {1,...,#3}{
    \draw [->] ($(\urx,\ury)+(-#5/2,0)-\r*(0,\rstep)$) -- ($(\urx,\ury)+(#5/2,0)-\r*(0,\rstep)$);}}\fi
}

\def\fillbox#1#2#3#4#5#6{
  \pgfgetlastxy{\llx}{\lly}
  \path #1;
  \pgfgetlastxy{\w}{\h}
  \pgfmathsetlengthmacro{\urx}{\llx+\w}
  \pgfmathsetlengthmacro{\ury}{\lly+\h}
  \filldraw [fill=gray!#6] (\llx,\lly) rectangle (\urx,\ury);
  \pgfmathsetlengthmacro{\xave}{(\llx+\urx)/2}
  \pgfmathsetlengthmacro{\yave}{\ury-8}
  \node at (\xave,\yave) {#4};
  \pgfmathsetlengthmacro{\ydiff}{\ury-\lly}
  \pgfmathsetlengthmacro{\lstep}{\ydiff/(#2+1)}
  \pgfmathsetlengthmacro{\rstep}{\ydiff/(#3+1)}
  \ifnum #2=0{}\else{ 
   \foreach \l in {1,...,#2}{
    \draw [->] ($(\llx,\lly)+(-#5/2,0)+\l*(0,\lstep)$) -- ($(\llx,\lly)+(#5/2,0)+\l*(0,\lstep)$);}}\fi
  \ifnum #3=0{}\else{
   \foreach \r in {1,...,#3}{
    \draw [->] ($(\urx,\ury)+(-#5/2,0)-\r*(0,\rstep)$) -- ($(\urx,\ury)+(#5/2,0)-\r*(0,\rstep)$);}}\fi
}

\def\directarc#1#2{
  \path #1;
  \pgfgetlastxy{\lx}{\ly}
  \path #2;
  \pgfgetlastxy{\rx}{\ry}
  \pgfmathsetlengthmacro{\xave}{(\lx+\rx)/2}
  \draw #1 .. controls (\xave,\ly) and (\xave,\ry) .. #2;
}

\def\loopright#1#2#3{
  \path #1;
  \pgfgetlastxy{\ux}{\uy}
  \path #2;
  \pgfgetlastxy{\lx}{\ly}
  \pgfmathsetlengthmacro{\maxx}{max(\ux,\lx)}
  \pgfmathsetlengthmacro{\farx}{\maxx+#3}
  \draw #1 .. controls (\farx,\uy) and (\farx,\ly) .. #2;
}

\def\loopleft#1#2#3{
  \path #1;
  \pgfgetlastxy{\ux}{\uy}
  \path #2;
  \pgfgetlastxy{\lx}{\ly}
  \pgfmathsetlengthmacro{\minx}{min(\ux,\lx)}
  \pgfmathsetlengthmacro{\farx}{\minx-#3}
  \draw #1 .. controls (\farx,\uy) and (\farx,\ly) .. #2;
}

\def\fancyarc#1#2#3#4{
  \path #1;
  \pgfgetlastxy{\ux}{\uy}
  \path #2;
  \pgfgetlastxy{\lx}{\ly}
  \pgfmathsetlengthmacro{\xave}{(\lx+\ux)/2}
  \pgfmathsetlengthmacro{\yave}{(\ly+\uy)/2+#4}
  \loopleft{#1}{(\xave,\yave)}{#3}
  \loopright{#2}{(\xave,\yave)}{#3}
}

\def\activetikz#1{$$\begin{tikzpicture}#1\end{tikzpicture}$$}
\def\inactivetikz#1{}

\newtheorem{thm}{Theorem}[section]
\newtheorem{lem}[thm]{Lemma}

\newtheorem{prop}[thm]{Proposition}

\theoremstyle{remark}
\newtheorem{rem}[thm]{Remark}

\theoremstyle{definition}
\newtheorem{defn}[thm]{Definition}
\newtheorem*{notation*}{Notation}

\theoremstyle{definition}
\newtheorem{ex}[thm]{Example}

\DeclareMathOperator{\id}{id}

\DeclareMathOperator{\Hom}{Hom}

\DeclareMathOperator{\inj}{\hookrightarrow}

\DeclareMathOperator{\Ob}{Ob}

\def\taking{\colon}
\def\too{\longrightarrow}
\newcommand{\To}[1]{\xrightarrow{#1}}
\newcommand{\From}[1]{\xleftarrow{#1}}
\def\iso{\cong}
\def\ss{\subseteq}
\def\bigid{\mathds{1}}
\newcommand{\ol}[1]{\overline{#1}}

\def\Set{\mathbf{Set}}
\def\Sets{\mathbf{Sets}}
\def\FinSet{\mathbf{FinSet}}

\def\ODS{\mathbf{ODS}}
\def\VF{\mathbf{VF}}
\newcommand{\TFS}[1]{\mathbf{TFS}_{#1}}
\def\op{^{\text{op}}}
\def\Man{\mathbf{Man}}
\def\Euc{\mathbf{Euc}}
\def\mcC{\mathcal{C}}
\def\mcD{\mathcal{D}}
\def\mcG{\mathcal{G}}
\def\mcL{\mathcal{L}}
\def\bfL{\mathbf{Lin}}
\def\bfW{\mathbf{W}}
\newcommand{\Opd}[1]{\mathcal{O}#1}

\def\NN{\mathbb N}

\def\RR{\mathbb R}

\newcommand{\inp}[1]{#1^{\text{in}}}
\newcommand{\outp}[1]{#1^{\text{out}}}
\newcommand{\expt}[2]{#1_{#2}^{\text{exp}}}
\newcommand{\loc}[2]{#1_{#2}^{\text{loc}}}

\usepackage[left=1.8in,right=1.8in,top=1.6in,bottom=1.4in]{geometry}

\title[Algebras of Open Systems on the Operad of Wiring Diagrams]{Algebras of Open Dynamical Systems on the Operad of Wiring Diagrams}

\author{Dmitry Vagner}
\author {David I. Spivak}
\author{Eugene Lerman}

\thanks{Spivak was supported by ONR grant N000141310260 and AFOSR grant FA9550-14-1-0031.}
\begin{document}

\begin{abstract}
In this paper, we use the language of operads to study open dynamical systems. More specifically, we study the algebraic nature of assembling complex dynamical systems from an interconnection of simpler ones. The syntactic architecture of such interconnections is encoded using the visual language of wiring diagrams. We define the symmetric monoidal category $\bfW$, from which we may construct an operad $\Opd{\bfW}$, whose objects are black boxes with input and output ports, and whose morphisms are wiring diagrams, thus prescribing the algebraic rules for interconnection. We then define two \mbox{$\bfW$-algebras} $\mcG$ and $\mcL$, which associate semantic content to the structures in $\bfW$. Respectively, they correspond to general and to linear systems of differential equations, in which an internal state is controlled by inputs and produces outputs. As an example, we use these algebras to formalize the classical problem of systems of tanks interconnected by pipes, and hence make explicit the algebraic relationships among systems at different levels of granularity.
\end{abstract}

\maketitle

\section{Introduction}
\label{sec:intro}

It is widely believed that complex systems of interest in the sciences and engineering are both modular and hierarchical. Network theory uses the tools and visual language of graph theory to model such systems, and has proven to be both effective and flexible in describing their modular character. However, the field has put less of an emphasis on finding powerful and versatile language for describing the hierarchical aspects of complex systems. There is growing confidence that category theory can provide the necessary conceptual setting for this project. This is seen, for example, in Mikhail Gromov's well-known claim, ``the mathematical language developed by the end of the 20th century by far exceeds in its expressive power anything, even imaginable, say, before 1960. Any meaningful idea coming from science can be fully developed in this language.'' \cite{Gromov}

Joyal and Street's work on string diagrams \cite{JSTensor} for monoidal categories and (with Verity) on traced monoidal categories \cite{JSTraced} has been used for decades to visualize compositions and feedback in networked systems, for example in the theory of flow charts \cite{Arthan}. Precursors, such as Penrose diagrams and flow diagrams, have been used in physics and the theory of computation, respectively, since the 1970's \cite{Scott,Baez1}.

Over the past several years, the second author and collaborators have been developing a novel approach to modular hierarchical systems based on the language of operads and symmetric monoidal categories \cite{Spivak2,RupelSpivak}. The main contribution to the theory of string diagrams of the present research program is the inclusion of an outer box, which allows for holarchic \cite{Koestler} combinations of these diagrams. That is, the parts can be assembled into a whole, which can itself be a part. The composition of such assemblies can now be viewed as morphism composition in an operad. In fact, there is a strong connection between traced monoidal categories and algebras on these operads, such as our operad $\Opd{\bfW}$ of wiring diagrams, though it will not be explained here (see \cite{SpivakSchultzRupel} for details). 

More broadly, category theory can organize graphical languages found in a variety of applied contexts. For example, it is demonstrated in \cite{Baez1} and \cite{Coecke} that the theory of monoidal categories unifies the diagrams coming from diverse fields such as physics, topology, logic, computation, and linguistics. More recently, as in \cite{Baez2}, there has been growing interest in viewing more traditionally applied fields, such as ecology, biology, chemistry, electrical engineering, and control theory through such a lens. Specifically, category theory has been used to draw connections among visual languages such as planar knot diagrams, Feynman diagrams, circuit diagrams, signal flow graphs, Petri nets, entity relationship diagrams, social networks, and flow charts. This research is building toward what John Baez has called ``a foundation of applied mathematics'' \cite{BaezTalk}.

\hspace{3 mm}

The goal of the present paper is to show that open continuous time dynamical systems form an algebra over a certain (colored) operad, which we call the operad of \emph{wiring diagrams}.  It is a variant of the operad that appeared in \cite{RupelSpivak}. That is, wiring diagrams provide a straightforward, diagrammatic language to understand how dynamical systems that describe processes can be built up from the systems that describe its sub-processes. 

More precisely, we will define a symmetric monoidal category $\bfW$ of black boxes and wiring diagrams. Its underlying operad $\Opd{\bfW}$ is a graphical language for building larger black boxes out of an interconnected set of smaller ones. We then define two $\bfW$-algebras, $\mcG$ and $\mcL$, which encode \emph{open dynamical systems}, i.e., differential equations of the form
\begin{align}\label{dia:basic form}
\begin{cases}
\dot{Q}=\inp{f}(Q,input)\\
output=\outp{f}(Q)
\end{cases}
\end{align}
where $Q$ represents an internal state vector, $\dot{Q}=\frac{dQ}{dt}$ represents its time derivative, and $input$ and $output$ represent inputs to and outputs from the system. In $\mcG$, the functions $\inp{f}$ and $\outp{f}$ are smooth, whereas in the subalgebra $\mcL\ss\mcG$, they are moreover linear. The fact that $\mcG$ and $\mcL$ are $\bfW$-algebras captures the fact that these systems are closed under wiring diagram interconnection. 

Our notion of interconnection is a generalization of that in Deville and Lerman \cite{DL1}, \cite{DL2}, \cite{DL3}. Their version of interconnection produces a closed system from open ones, and can be understood in the present context as a morphism whose codomain is the closed box (see Definition~\ref{def:mon}). Graph fibrations between wiring diagrams form an important part of their  formalism, though we do not discuss that aspect here.

This paper is the third in a series, following \cite{RupelSpivak} and \cite{Spivak2}, on using wiring diagrams to model interactions. The algebra we present here, that of open systems, is distinct from the algebras of relations and of propagators studied in earlier works. Beyond the dichotomy of discrete vs. continuous, these algebras are markedly different in structure. For one thing, the internal wires in \cite{RupelSpivak} themselves carry state, whereas here, a wire should be thought of as instantaneously transmitting its contents from an output site to an input site. Another difference between our algebra and those of previous works is that the algebras here involve \emph{open systems} in which, as in (\ref{dia:basic form}), the instantaneous change of state is a function of the current state and the input, whereas the output depends only on the current state (see Definition~\ref{def:general algebra}). The differences between these algebras is also reflected in a mild difference between the operad we use here and the one used in previous work. 

\subsection{Motivating example}

The motivating example for the algebras in this paper comes from classical differential equations pedagogy; namely, systems of tanks containing salt water concentrations, with pipes carrying fluid among them. The systems of ODEs produced by such applications constitute a subset of those our language can address; they are linear systems with a certain form (see Example~\ref{ex:as promised}). To ground the discussion, we consider a specific example.

\begin{ex} \label{ex:main}
Figure~\ref{fig:pipebrine} below reimagines a problem from Boyce and DiPrima's canonical text \mbox{\cite[Figure 7.1.6]{BD}} as a dynamical system over a \emph{wiring diagram}.

\begin{figure}[ht]
\activetikz{
	\path(0,0);
	\blackbox{(10,5)}{2}{1}{$Y$}{.7}
	        \node at (.4,3.6) {\small $\inp{Y}_{a}$};
	        \node at (.4,1.9) {\small $\inp{Y}_{b}$};
	        \node at (9.6,2.75) {\small $\outp{Y}_a$};
	\path(2,1.5);
	\blackbox{(2,2)}{2}{1}{$X_1$}{.5}
	        \node at (3,2.6) {\tiny $Q_1(t)$ oz salt};
	        \node at (3,2.3) {\tiny 30 gal water};
	        \node at (1.7,3.06) {\small $\inp{X}_{1a}$};
	        \node at (1.7,2.4) {\small $\inp{X}_{1b}$};
	        \node at (4.38,2.7) {\small $\outp{X}_{1a}$};
	\path(6,1.5);
	\blackbox{(2,2)}{2}{2}{$X_2$}{.5}
	        \node at (7,2.6) {\tiny $Q_2(t)$ oz salt};
	        \node at (7,2.3) {\tiny 20 gal water};
	        \node at (5.7,3.07) {\small $\inp{X}_{2a}$};
	        \node at (5.7,2.4) {\small $\inp{X}_{2b}$};
	        \node at (8.37,3.07) {\small $\outp{X}_{2a}$};
	        \node at (8.37,2.37) {\small $\outp{X}_{2b}$};
	\directarc{(4.25,2.5)}{(5.75,2.16667)} 
	    \node at (5,2) {\tiny 3 gal/min};
	\directarc{(0.35,1.6667)}{(1.75,2.83333)} 
	    \node at (.7,1.4) {\tiny 1.5 gal/min};
	    \node at (.7,1.2) {\tiny 1 oz/gal};
	\fancyarc{(0.35,3.3333)}{(5.75,2.83333)}{-40}{25} 
	    \node at (.6,3.1) {\tiny 1 gal/min};
	    \node at (.6,2.9) {\tiny 3 oz/gal};
	\directarc{(8.25,2.8333)}{(9.65,2.5)} 
	    \node at (9.5,2.3) {\tiny 2.5};
	    \node at (9.5,2.1) {\tiny gal/min};
	\fancyarc{(1.75,2.16667)}{(8.25,2.16667)}{20}{-45} 
	    \node at (5,.3) {\tiny 1.5 gal/min};
}
\caption{A dynamical system from Boyce and DiPrima interpreted over a wiring diagram $\Phi\taking X_1,X_2\to Y$ in $\Opd{\bfW}$.}
\label{fig:pipebrine}
\end{figure}

In this diagram, $X_1$ and $X_2$ are boxes that represent tanks consisting of salt water solution. The functions $Q_1(t)$ and $Q_2(t)$ represent the amount of salt (in ounces) found in 30 and 20 gallons of water, respectively. These tanks are interconnected with each other by pipes embedded within a total system $Y$. The prescription for how wires are attached among the boxes is formally encoded in the wiring diagram $\Phi:X_1,X_2\to Y$, as we will discuss in Definition \ref{def:W}.

Both tanks are being fed salt water concentrations at constant rates from the outside world. Specifically, $X_1$ is fed a 1 ounce salt per gallon water solution at 1.5 gallons per minute and $X_2$ is fed a 3 ounce salt per gallon water solution at 1 gallon per minute. The tanks also both feed each other their solutions, with $X_1$ feeding $X_2$ at 3 gallons per minute and $X_2$ feeding $X_1$ at 1.5 gallons per minute. Finally, $X_2$ feeds the outside world its solution at 2.5 gallons per minute. 

The dynamics of the salt water concentrations both within and leaving each tank $X_i$ is encoded in a linear open system $f_i$, consisting of a differential equation for $Q_i$ and a readout map for each $X_i$ output (see Definition \ref{def:opensystem}). Our algebra $\mcL$ allows one to assign a linear open system $f_i$ to each tank $X_i$, and by functoriality the morphism $\Phi\taking X_1,X_2\to Y$ produces a linear open system for the larger box $Y$. We will explore this construction in detail, in particular providing explicit formulas for it in the linear case, as well as for more general systems of ODEs.
\end{ex}

\section{Preliminary Notions}
\label{sec:pre}

Throughout this paper we use the language of monoidal categories and functors. Depending on the audience, appropriate background on basic category theory can be found in MacLane \cite{MacLane}, Awodey \cite{Awodey}, or Spivak \cite{Spivak}. Leinster~\cite{Leinster} is a good source for more specific information on monoidal categories and operads. We refer the reader to \cite{KFA} for an introduction to dynamical systems.

\begin{notation*}
We denote the category of sets and functions by $\Set$ and the full subcategory spanned by finite sets as $\FinSet$. We generally do not concern ourselves with cardinality issues. We follow Leinster \cite{Leinster} and use $\times$ for binary product and $\Pi$ for arbitrary product, and dually $+$ for binary coproduct and $\amalg$ for arbitrary coproduct in any category. By {\em operad} we always mean symmetric colored operad or, equivalently, symmetric multicategory.
\end{notation*}

\subsection{Monoidal categories and operads} In Section~\ref{sec:W}, we will construct the symmetric monoidal category $(\bfW,\oplus,0)$ of boxes and wiring diagrams, which we often simply denote as $\bfW$. We will sometimes consider the underlying operad $\Opd{\bfW}$, obtained by applying the fully faithful functor 
\[\mathcal{O}\taking \mathbf{SMC}\to\mathbf{Opd}\]
to $\bfW$. A brief description of this functor $\Opd{}$ is given below in Definition~\ref{def:SMC to Opd}.

\begin{defn}\label{def:SMC to Opd}

Let $\mathbf{SMC}$ denote the category of symmetric monoidal categories and lax monoidal functors; and $\mathbf{Opd}$ be the category of operads and operad functors. Given a symmetric monoidal category $(\mcC,\otimes,I _\mcC)\in\Ob\mathbf{SMC}$, we define the operad $\Opd{\mcC}$ as follows:
\[\Ob \Opd{\mcC}:=\Ob\mcC, \hspace{10 mm} \Hom_{\Opd{\mcC}} (X_1,\ldots,X_n;Y):=\Hom_{\mcC}(X_1\otimes\cdots\otimes X_n,Y)\]
for any $n\in\NN$ and objects $X_1,\ldots,X_n,Y\in\Ob\mcC$.

Now suppose $F\taking (\mcC,\otimes,I _{\mcC})\to(\mcD,\odot,I _{\mcD})$ is a lax monoidal functor in $\mathbf{SMC}$. By definition such a functor is equipped with a morphism
\[\mu\taking FX_1\odot\cdots\odot FX_n\to F(X_1\otimes\cdots\otimes X_n),\]
natural in the $X_i$, called the {\em coherence map}. With this map in hand, we define the operad functor $\Opd{F}\taking \Opd{\mcC}\to \Opd{\mcD}$ by stating how it acts on objects $X$ and morphisms $\Phi\taking X_1,\ldots,X_n\to Y$ in $\Opd{\mcC}$:
\[\Opd{F}(X):=F(X),\hspace{1.8 mm} \Opd{F}(\Phi:X_1,\ldots,X_n\to Y):=F(\Phi)\circ\mu:FX_1\odot\cdots\odot FX_n\to FY.\]
\end{defn}

\begin{ex} \label{ex:sets} Consider the symmetric monoidal category $(\Set ,\times,\star)$, where $\times$ is the cartesian product of sets and $\star$ a one element set. Define \mbox{$\Sets:=\Opd{\Set }$} as in Definition \ref{def:SMC to Opd}. Explicitly, $\Sets$ is the operad in which an object is a set and a morphism $f\taking X_1,\ldots,X_n\to Y$ is a function $f\taking X_1\times\cdots\times X_n\to Y$.
\end{ex}

\begin{defn} \label{def:algebra}
Let $\mcC$ be a symmetric monoidal category and let $\Set=(\Set,\times,\star)$ be as in Example \ref{ex:sets}. A \emph{$\mcC$-algebra} is a lax monoidal functor $\mcC\to\Set$. Similarly, if $\mcD$ is an operad, a \emph{$\mcD$-algebra} is defined as an operad functor $\mcD\to\Sets$. 
\end{defn}

To avoid subscripts, we will generally use the formalism of SMCs in this paper. Definitions~\ref{def:SMC to Opd} and \ref{def:algebra} can be applied throughout to recast everything we do in terms of operads. The primary reason operads may be preferable in applications is that they suggest more compelling pictures. Hence throughout this paper, depictions of wiring diagrams will often be operadic, i.e., have many input boxes wired together into one output box.

\subsection{Typed sets} Each box in a wiring diagram will consist of finite sets of ports, each labelled by a type. To capture this idea precisely, we define the notion of typed finite sets. By a \emph{finite product} category, we mean a category that is closed under taking finite products.

\begin{defn}\label{def:typed finite sets} 
Let $\mcC$ be a small finite product category. The category of \emph{$\mcC$-typed finite sets}, denoted $\TFS{\mcC}$, is defined as follows. An object in $\TFS{\mcC}$ is a map from a finite set to the objects of $\mcC$: 
\[\Ob\TFS{\mcC}:=\{(A,\tau)\; |\; A\in\Ob\FinSet, \tau\taking A\to\Ob\mcC)\}.\] 
Intuitively, one can think of a typed finite set as a finite unordered list of $\mcC$-objects. For any element $a\in A$, we call the object $\tau(a)$ its {\em type}. If the typing function $\tau$ is clear from context, we may denote $(A,\tau)$ simply by $A$.

A morphism $q\taking(A,\tau)\to (A',\tau')$ in $\TFS{\mcC}$ consists of a function $q\taking A\to A'$ that makes the following diagram of finite sets commute:
\[\xymatrix{
A \ar[rr]^q \ar[rd]_\tau
& {}
& A' \ar[ld]^{\tau'}\\
&\Ob\mcC
} \]
Note that $\TFS{\mcC}$ is a cocartesian monoidal category.

We refer to the morphisms of $\TFS{\mcC}$ as {\em $\mcC$-typed functions}. If a $\mcC$-typed function $q$ is bijective, we call it a \emph{$\mcC$-typed bijection}.
\end{defn}

In other words, $\TFS{\mcC}$ is the comma category for the diagram 
$$\FinSet\To{i}\Set\From{\Ob\mcC}\{*\}$$
where $i$ is the inclusion.

\begin{defn} \label{def:depprod}
Let $\mcC$ be a finite product category, and let $(A,\tau)\in\Ob\TFS{\mcC}$ be a $\mcC$-typed finite set. Its \emph{dependent product} $\ol{(A,\tau)}\in\Ob\mcC$ is defined as 
\[\overline{(A,\tau)}:=\prod_{a\in A}\tau(a).\] 
Coordinate projections and diagonals are generalized as follows. Given a typed function $q\taking (A,\tau)\to (A',\tau')$ in $\TFS{\mcC}$ we define 
\[\ol{q}\taking \ol{(A',\tau')}\to\ol{(A,\tau)}\]
to be the unique morphism for which the following diagram commutes for all \mbox{$a\in A$}:
\[\xymatrix{
\prod_{a'\in A'}\tau'(a') \ar[r]^{\overline{q}} \ar[d]_{\pi_{q(a)}}
& \prod_{a\in A}\tau(a) \ar[d]^{\pi_a}\\
\tau'(q(a))\ar@{=}[r]&\tau(a)
} \]
By the universal property for products, this defines a functor, 
\[\ol{\;\cdot\;}\taking\TFS{\mcC}\op\to\mcC.\]
\end{defn}

\begin{lem} The dependent product functor $\TFS{\mcC}\op\to\mcC$ is strong monoidal. In particular, for any finite set $I$ whose elements index typed finite sets $(A_i,\tau_i)$, there is a canonical isomorphism in $\mcC$, 
$$\ol{\coprod_{i\in I}(A_i,\tau_i)}\iso\prod_{i\in I}\ol{(A_i,\tau_i)}.$$
\end{lem}

\begin{rem} \label{rem:default} The category of second-countable smooth manifolds and smooth maps is essentially small (by the embedding theorem) so we choose a small representative and denote it $\Man$. Note that $\Man$ is a finite product category. Manifolds will be our default typing, in the sense that we generally take $\mcC:=\Man$ in Definition \ref{def:typed finite sets} and denote
\begin{align}\label{dia:TFS}
\TFS{}:=\TFS{\Man}.
\end{align}
We thus refer to the objects, morphisms, and isomorphisms in $\TFS{}$ simply as \emph{typed finite sets}, \emph{typed functions}, and \emph{typed bijections}, respectively.
\end{rem}

\begin{rem} \label{rem:TFSL}
The ports of each box in a wiring diagram will be labeled by manifolds because they are the natural setting for geometrically interpreting differential equations (see \cite{SpiM-CalcMan}). For simplicity, one may wish to restrict attention to the full subcategory $\Euc$ of Euclidean spaces $\RR^n$ for $n\in\NN$, because they are the usual domains for ODEs found in the literature; or to the (non-full) subcategory $\bfL$ of Euclidean spaces and linear maps between them, because they characterize linear systems of ODEs. We will return to $\TFS{\bfL}$ in Section \ref{sec:l}.
\end{rem}

\subsection{Open systems} As a final preliminary, we define our notion of open dynamical system. Recall that every manifold $M$ has a {\em tangent bundle} manifold, denoted $TM$, and a smooth projection map $p\taking TM\to M$. For any point $m\in M$, the preimage $T_mM:=p^{-1}(m)$ has the structure of a vector space, called the {\em tangent space of $M$ at $m$}. If $M\iso\RR^n$ is a Euclidean space then also $T_mM\iso\RR^n$ for every point $m\in M$. A {\em vector field on $M$} is a smooth map $g\taking M\to TM$ such that $p\circ g=\id_M$. See \cite{SpiM-CalcMan} or \cite{Warner} for more background. 

For the purposes of this paper we make the following definition of open systems; this may not be completely standard.

\begin{defn}\label{def:opensystem}
Let $M,\inp{U},\outp{U}\in\Ob\Man$ be smooth manifolds and $TM$ be the tangent bundle of $M$. Let $f=(\inp{f},\outp{f})$ denote a pair of smooth maps
\begin{align*}
\begin{cases}
\inp{f}\taking M\times\inp{U}\to TM\\
\outp{f}\taking M\to\outp{U}
\end{cases}
\end{align*}
where, for all $(m,u)\in M\times\inp{U}$ we have $\inp{f}(m,u)\in T_mM$; that is, the following diagram commutes:
\[ \xymatrix{
M\times \inp{U} \ar[rr]^{\inp{f}} \ar[rd]_{\pi_M}
& {}
& TM \ar[ld]^{p}\\
& M
} \]
We sometimes use $f$ to denote the whole tuple, 
$$f=(M,\inp{U},\outp{U},f),$$ 
which we refer to as an \emph{open dynamical system} (or \emph{open system} for short). We call $M$ the {\em state space}, $\inp{U}$ the \textit{input space}, $\outp{U}$ the \textit{output space}, $\inp{f}$ the \emph{differential equation}, and $\outp{f}$ the \emph{readout map} of the open system. 

Note that the pair $f=(\inp{f},\outp{f})$ is determined by a single smooth map
$$f\taking M\times\inp{U}\to TM\times\outp{U},$$
which, by a minor abuse of notation, we also denote by $f$. 

In the special case that $M,U^\text{in},U^\text{out}\in\Ob\bfL$ are Euclidean spaces and $f$ is a linear map (or equivalently $\inp{f}$ and $\outp{f}$ are linear), we call $f$ a \emph{linear open system}.

\begin{rem}\label{ex:dynamical system} 
  Let $M$ be a smooth manifold, and let
  \mbox{$\inp{U}=\outp{U}=\RR^0$} be trivial. Then an open system in
  the sense of Definition~\ref{def:opensystem} is a smooth map
  \mbox{$f\taking M\to TM$} over $M$, in other words, a vector field
  on $M$. From the geometric point of view, vector fields are
  autonomous (i.e., closed!) dynamical systems; see~\cite{Teschl}.  
\end{rem}
\begin{rem}
  For an arbitrary manifold $\inp{U}$, a map \mbox{$M\times\inp{U}\to
    TM$} can be considered as a function $\inp{U}\to\VF(M)$, where
  $\VF(M)$ is the set of vector fields on $M$. Hence, $\inp{U}$ {\em
    controls} the behavior of the system in the usual sense.
\end{rem}
%
%
%
%
\end{defn}

\begin{rem} \label{rem:thepoint} Given an open system $f$ we can form
  a new open system by feeding the readout of $f$ into the inputs of
  $f$.
For example suppose the open system is of the form
\[
\begin{cases}
M\times A\times B \xrightarrow{F} TM \\
g= (g_A, g_B)\colon M \to C\times B,
\end{cases}
\]
where $A$, $B$, $C$ and $M$ are manifolds.
Define $F'\colon M\times A\to TM$ by
\[
F'(m,a) := F(m, a, g_B(m))\qquad \textrm{ for all }\quad (m,a)\in M\times A.
\] 
Then
\[
\begin{cases}
M\times A \xrightarrow{F'} TM \\
g_A\colon M \to C
\end{cases}
\]
is a new open system obtained by plugging a readout of $f$ into the
space of inputs $B$.  Compare with Figure~\ref{fig:wiringdiagram}.

This looks a little boring.  
It becomes more interesting when we start with several open systems,
take their product and then plug (some of the) outputs into inputs.
For example suppose we start with two open systems
\[
\begin{cases}
M_1\times A\times B \xrightarrow{F_1} TM_1 \\
g_1\colon M_1 \to C
\end{cases}
\]
and
\[
\begin{cases}
M_2\times C \xrightarrow{F_2} TM_2 \\
g_2 = (g_B,g_D)\colon M_2 \to B\times D
\end{cases}.
\]
Here, again, all capital letters denote manifolds.
Take their product; we get
\[
\begin{cases}
M_1\times A\times B \times M_2\times C 
\xrightarrow{(F_1,F_2)} TM_1\times TM_2 \\
(g_1,g_2)\colon M_1\times M_2 \to C\times B\times D
\end{cases}
\]
Now plug in the functions $g_B$ and $g_1$ into inputs.  We get a new system
\[
 \begin{cases}
M_1\times  M_2\times A \xrightarrow{F'} TM_1\times TM_2 \\
g'\colon M_1\times M_2 \to  D
\end{cases}
\]
where 
\[
F'(m_1, m_2, a):= (F_1(m_1, a, g_B(m_2)), F_2 (m_2, g_1(m_1)).
\]
 Compare with Figure~\ref{fig:WD}.
 Making these kinds of operations on open systems precise for an arbitrary number of interacting systems is the point
 of our paper.
\end{rem}


By defining the appropriate morphisms, we can consider open dynamical systems as being objects in a category. We are not aware of this notion being defined previously in the literature, but it is convenient for our purposes.

\begin{defn} \label{def:odscat} Suppose that $M_i,\inp{U}_i,\outp{U}_i\in\Ob\Man$ and $(M_i,\inp{U}_i,\outp{U}_i,f_i)$ is an open system for \mbox{$i\in\{1,2\}$}. A \emph{morphism of open systems} 
\[\zeta\taking(M_1,\inp{U}_1,\outp{U}_1,f_1)\to (M_2,\inp{U}_2,\outp{U}_2,f_2)\]
is a triple $(\zeta_M,\zeta_{\inp{U}},\zeta_{\outp{U}})$ of smooth maps $\zeta_M\taking M_1\to M_2$, $\zeta_{\inp{U}}\taking \inp{U}_1\to \inp{U}_2$, and $\zeta_{\outp{U}}\taking \outp{U}_1\to \outp{U}_2$, such that the following diagram commutes:

\[\xymatrix{
M_1\times\inp{U}_1 \ar[r]^{f_1} \ar[d]_{\zeta_M\times\zeta_{\inp{U}}}
&TM_1\times\outp{U}_1  \ar[d]^{T\zeta_M\times\zeta_{\outp{U}}}\\
M_2\times\inp{U}_2 \ar[r]_{f_2}
& TM_2\times\outp{U}_2
} \]

This defines the category $\ODS $ of open dynamical systems. We define the subcategory $\ODS _\bfL\ss\ODS$ by restricting our objects to linear open systems, as in Definition~\ref{def:opensystem}, and imposing that the three maps in $\zeta$ are linear. 
\end{defn}

As in Remark~\ref{rem:thepoint}, we will often want to combine two or more interconnected open systems into one larger one. As we shall see in Section~\ref{sec:g}, this will involve taking a product of the smaller open systems. Before we define this formally, we first remind the reader that the tangent space functor $T$ is strong monoidal, i.e., it canonically preserves products, 
$$T(M_1\times M_2)\iso TM_1\times TM_2.$$

\begin{lem} \label{def:osprod}
The category $\ODS$ of open systems has all finite products. That is, if $I$ is a finite set and $f_i=(M_i,\inp{U}_i,\outp{U}_i,f_i)\in\Ob\ODS$ is an open system for each $i\in I$, then their product is 
$$\prod_{i\in I}f_i=\left(\prod_{i\in I}M_i,\prod_{i\in I}\inp{U}_i,\prod_{i\in I}\outp{U}_i,\prod_{i\in I}f_i\right)$$
with the obvious projection maps.
\end{lem}

\section{The Operad of Wiring Diagrams}
\label{sec:W}

In this section, we define the symmetric monoidal category $(\bfW,\oplus,0)$ of wiring diagrams. We then use Definition~\ref{def:SMC to Opd} to define the wiring diagram operad $\Opd{\bfW}$, which situates our pictorial setting. We begin by formally defining the underlying category $\bfW$ and continue with some concrete examples to explicate this definition.

\begin{defn}\label{def:W} 
The category $\bfW$ has objects \emph{boxes} and morphisms \emph{wiring diagrams}. A box $X$ is an ordered pair of $\Man$-typed finite sets (Definition~\ref{def:typed finite sets}), 
\[X=(\inp{X},\outp{X})\in\Ob\TFS{}\times\Ob\TFS{}.\]
Let $\inp{X}=(A,\tau)$ and $\outp{X}=(A',\tau')$. Then we refer to elements $a\in A$ and $a'\in A'$ as \emph{input ports} and \textit{output ports}, respectively. We call $\tau(a)\in\Ob\Man$ the \emph{type} of port $a$, and similarly for $\tau'(a')$.

A wiring diagram $\Phi\taking X\to Y$ in $\bfW$ is a triple $(X,Y,\varphi)$, where $\varphi$ is a typed bijection (see Definition \ref{def:typed finite sets})
\begin{align}\label{dia:wd function}
\varphi\taking\inp{X}+\outp{Y}\xrightarrow{\cong} \outp{X}+\inp{Y},
\end{align}
satisfying the following condition:
\begin{description}
\item[no passing wires] $\varphi(\outp{Y})\cap\inp{Y}=\varnothing$, or equivalently $\varphi(\outp{Y})\ss\outp{X}$. 
\end{description}
This condition allows us to decompose $\varphi$ into a pair $\varphi=(\inp{\varphi},\outp{\varphi})$:
\begin{align}\label{dia:components of wd}
   \left\{
     \begin{array}{l}
       \inp{\varphi}\taking \inp{X} \to \outp{X}+\inp{Y} \\
       \outp{\varphi}\taking \outp{Y} \to \outp{X}
     \end{array}
   \right.
\end{align}

We often identify the wiring diagram $\Phi=(X,Y,\varphi)$ with the typed bijection $\varphi$, or equivalently its corresponding pair $(\inp{\varphi},\outp{\varphi})$.

By a \emph{wire} in $\Phi$, we mean a pair $(a,b)$, where $a\in\inp{X}+\outp{Y}$, $b\in\outp{X}+\inp{Y}$, and $\varphi(a)=b$. In other words a wire in $\Phi$ is a pair of ports connected by $\phi$.

The \emph{identity} wiring diagram $\iota:X\to X$ is given by the identity morphism $\inp{X}+\outp{X}\to\inp{X}+\outp{X}$ in $\TFS{}$.

Now suppose $\Phi=(X,Y,\varphi)$ and $\Psi=(Y,Z,\psi)$ are wiring diagrams. We define their \emph{composition} as $\Psi\circ\Phi=(X,Z,\omega)$, where $\omega=(\inp{\omega},\outp{\omega})$ is given by the pair of dashed arrows making the following diagrams commute. 
\begin{equation}\label{dia:composition diagrams}
\xymatrixcolsep{3.5pc}
\xymatrix{
\inp{X} 
    \ar[dd]_{\inp{\varphi}} 
    \ar@{-->}[r]^{\inp{\omega}} &
\outp{X}+\inp{Z} \\
{} & 
\outp{X}+\outp{X}+\inp{Z}
    \ar[u]_{\nabla+\bigid _{\inp{Z}}} \\
\outp{X}+\inp{Y}
    \ar[r]_-{\bigid _{\outp{X}}+\inp{\psi}} & 
\outp{X}+\outp{Y}+\inp{Z}
    \ar[u]_{\bigid _{\outp{X}}+\outp{\varphi}+\bigid _{\inp{Z}}}
} \mskip5mu
\xymatrixcolsep{1.5pc}
\xymatrix{
\outp{Z} 
    \ar[rd]_{\outp{\psi}} 
    \ar@{-->}[rr]^{\outp{\omega}}
&
&\outp{X}\\
&\outp{Y} 
    \ar[ru]_{\outp{\varphi}}}
\end{equation}
Here $\nabla\taking \outp{X}+\outp{X}\to \outp{X}$ is the codiagonal map in $\TFS{}$.
\end{defn}

\begin{rem}\label{rem:different Cs}
For any finite product category $\mcC$, we may define the category $\bfW_{\mcC}$ by replacing $\Man$ with $\mcC$, and $\TFS{}$ with $\TFS{\mcC}$, in Definition~\ref{def:W}. In particular, as in Remark~\ref{rem:TFSL}, we have the symmetric monoidal category $\bfW_{\bfL}$ of linearly typed wiring diagrams.
\end{rem}

What we are calling a box is nothing more than an interface; at this stage it has no semantics, e.g., in terms of differential equations. Each box can be given a pictorial representation, as in Example~\ref{ex:pictorial box} below.

\begin{ex} \label{ex:pictorial box}
As a convention, we depict a box $X=(\{a,b\},\{c\})$ with input ports connecting on the left and output ports connecting on the right, as in Figure~\ref{fig:box} below. When types are displayed, we label ports on the exterior of their box and their types adjacently on the interior of the box with a `:' symbol in between to designate typing. Reading types off of this figure, we see that the type of input port $a$ is the manifold $\RR$, that of input port $b$ is the circle $S^1$, and that of output port $c$ is the torus $T^2$.

\begin{figure}[!ht]
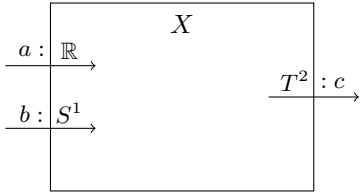

\activetikz{
	\path(0,0);\blackbox{(3.5,2.5)}{2}{1}{$X$}{1.2}
	\node at (-.25,1.8666) {\small $a:$};
	\node at (.25,1.8666) {\small $\RR$};
	\node at (-.25,1.0333) {\small $b:$};
	\node at (.25,1.0333) {\small $S^1$};
	\node at (3.75,1.4333) {\small $:c$};
	\node at (3.25,1.4666) {\small $T^2$};
}
\caption{A box with two input ports, of types $\RR$ and $S^1$, and one output port with type $T^2$.}
\label{fig:box}
\end{figure}
\end{ex}


A morphism in $\bfW$ is a wiring diagram $\Phi=(X,Y,\varphi)$, the idea being that a smaller box $X$ (the domain) is nested inside of a larger box $Y$ (the codomain). The ports of $X$ and $Y$ are then interconnected by wires, as specified by the typed bijection $\varphi$. We will now see an example of a wiring diagram, accompanied by a picture.

\begin{ex}\label{ex:wiringdiagram} Reading off the wiring diagram $\Phi=(X,Y,\varphi)$ drawn below in Figure~\ref{fig:wiringdiagram}, we have the following data for boxes: \[\begin{matrix} \inp{X}=\{a,b\} & \outp{X}=\{c,d\} \\ \inp{Y}=\{m\} & \outp{Y}=\{n\}\end{matrix}\]
Table~\ref{tab:wiringdiagram} makes $\varphi$ explicit via a list of its wires, i.e., pairs $(\gamma,\varphi(\gamma))$.

\noindent\begin{minipage}{\linewidth}
\[
\begin{array}{c||c|c|c}
\rule[-4pt]{0pt}{16pt}
\gamma\in\inp{X}+\outp{Y}& a & b & n 
\\\hline
\rule[-4pt]{0pt}{16pt}
\varphi(\gamma)\in\outp{X}+\inp{Y} & m & d & c
\end{array}
\]
\smallskip
\captionof{table}{} \label{tab:wiringdiagram} 
\end{minipage}

\vspace{-5 mm}

\begin{figure}[ht]
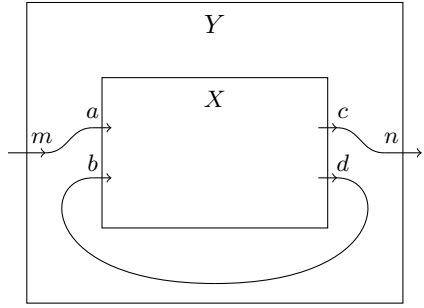

\activetikz{
	\path(0,0);\blackbox{(5,4)}{1}{1}{$Y$}{.5}
	\node at (0.2,2.2) {\small $m$};
	\node at (4.85,2.2) {\small $n$};
	\path (1,1);\blackbox{(3,2)}{2}{2}{\small$X$}{.25}
	\node at (0.875,2.533) {\small $a$};
	\node at (0.875,1.866) {\small $b$};
	\node at (4.2,2.533) {\small $c$};
	\node at (4.2,1.866) {\small $d$};
	\directarc{(.25,2)}{(.875,2.3333)}
	\directarc{(4.125,2.3333)}{(4.75,2)}
	\fancyarc{(.875,1.6666)}{(4.125,1.6666)}{20}{-40}
}
\caption{A Wiring Diagram $\Phi=(X,Y,\varphi)$.}
\label{fig:wiringdiagram}
\end{figure}
\end{ex}


\begin{rem} \label{rem:samestate}
The condition that $\varphi$ be typed, as in Definition \ref{def:typed finite sets}, ensures that if two ports are connected by a wire then the associated types are the same. In particular, in Example~\ref{ex:wiringdiagram} above, $(a,b,n)$ must be the same type tuple as $(m,d,c)$.
\end{rem}

Now that we have made wiring diagrams concrete and visual, we can do the same for their composition.

\begin{ex}\label{ex:composition}
In Figure~\ref{fig:compose}, we visualize the composition of two wiring diagrams $\Phi=(X,Y,\varphi)$ and $\Psi=(Y,Z,\psi)$ to form $\Psi\circ\Phi=(X,Z,\omega)$. Composition is depicted by drawing the wiring diagram for $\Psi$ and then, inside of the $Y$ box, drawing in the wiring diagram for $\Phi$. Finally, to depict the composition $\Psi\circ\Phi$ as one single wiring diagram, one simply ``erases" the $Y$ box, leaving the $X$ and $Z$ boxes interconnected among themselves. Figure~\ref{fig:compose} represents such a procedure by depicting the $Y$ box with a dashed arrow.

It's important to note that the wires also connect, e.g. if a wire in $\Psi$ connects a $Z$ port to some $Y$ port, and that $Y$ port attaches via a $\Phi$ wire to some $X$ port, then these wires ``link together" to a total wire in $\Psi\circ\Phi$, connecting a $Z$ port with an $X$ port. Table~\ref{tab:compose} below traces the wires of $\Psi\circ\Phi$ through the $\inp{\omega}$ and $\outp{\omega}$ composition diagrams in (\ref{dia:composition diagrams}) on its left and right side, respectively. The left portion of the table starts with $\gamma\in\inp{X}$ and ends at $\inp{\omega}(\gamma)\in\outp{X}+\inp{Z}$, with intermediary steps of the composition denoted with superscripts $\gamma^n$. The right portion of the table starts with $\gamma\in\outp{Z}$ then goes through the intermediary of $\gamma'\in\outp{Y}$ and finally reaches $\outp{\omega}(\gamma)\in\outp{Z}$. We skip lines on the right portion to match the spacing on the left.

\noindent\begin{minipage}{\linewidth}
\[
\begin{array}{c||c|c|c||c||c}
\rule[-4pt]{0pt}{16pt}
\gamma\in\inp{X}& a & b & c & v & \gamma\in\outp{Z}
\\\hline
\rule[-4pt]{0pt}{16pt}
\gamma^1\in\outp{X}+\inp{Y} & d & k & l & {} & {}
\\\hline
\rule[-4pt]{0pt}{16pt}
\gamma^2\in\outp{X}+\outp{Y}+\inp{Z} & d & u & n & m & \gamma'\in\outp{Y}
\\\hline
\rule[-4pt]{0pt}{16pt}
\gamma^3\in\outp{X}+\outp{X}+\inp{Z} & d & u & f & {} & {}
\\\hline
\rule[-4pt]{0pt}{16pt}
\inp{\omega}(\gamma)\in\outp{X}+\inp{Z} & d & u & f & e & \outp{\omega}(\gamma)\in\outp{X}
\end{array}
\]
\smallskip
\captionof{table}{} \label{tab:compose} 
\end{minipage}


\begin{figure}[ht]
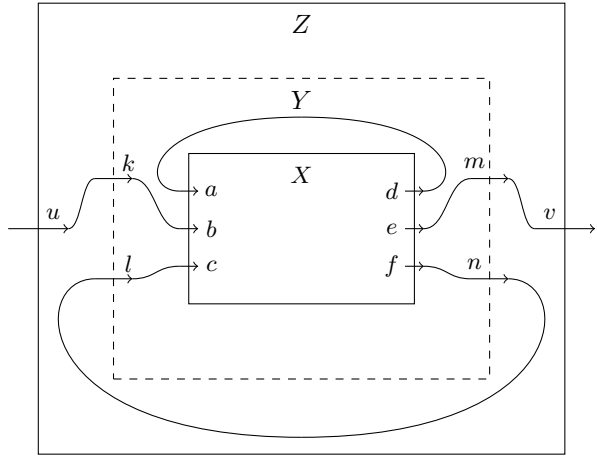

\activetikz{
    \path (-1,-1); \blackbox{(7,6)}{1}{1}{$Z$}{.8}
    \node at (-0.8,2.2) {\small $u$};
	\node at (5.8,2.2) {\small $v$};
	\path(0,0);\dashbox{(5,4)}{2}{2}{$Y$}{.5}
	\node at (0.2,2.88) {\small $k$};
	\node at (4.8,2.88) {\small $m$};
	\node at (0.2,1.53) {\small $l$};
	\node at (4.8,1.53) {\small $n$};
	\path (1,1);\blackbox{(3,2)}{3}{3}{\small$X$}{.25}
	\node at (1.3,1.5) {\small $c$};
	\node at (1.3,2) {\small $b$};
	\node at (1.3,2.5) {\small $a$};
	\node at (3.7,1.5) {\small $f$};
	\node at (3.7,2) {\small $e$};
	\node at (3.7,2.5) {\small $d$};
	\directarc{(-.6,2)}{(-.25,2.666)}
	\directarc{(5.25,2.666)}{(5.6,2)}
	\directarc{(0.25,1.333)}{(0.875,1.5)}
	\directarc{(4.125,1.5)}{(4.75,1.333)}
	\directarc{(0.25,2.666)}{(0.875,2)}
	\directarc{(4.125,2)}{(4.75,2.666)}
	\fancyarc{(-0.25,1.333)}{(5.25,1.333)}{25}{-60}
	\fancyarc{(0.875,2.5)}{(4.125,2.5)}{15}{28}
}
\caption{A wiring diagram composition $\Psi\circ\Phi=(X,Z,\omega)$ of $\Phi=(X,Y,\varphi)$ and $\Psi=(Y,Z,\psi)$, with dashed medium box $Y$.}
\label{fig:compose}
\end{figure}
\end{ex}

\begin{rem} \label{rem:pathology} 
The condition that $\varphi$ be both injective and surjective prohibits {\em exposed} ports and {\em split} ports, respectively, as depicted in Figure~\ref{fig:unsafe}{\bf a}. The {\em no passing wires} condition on $\varphi(\outp{Y})$ prohibits wires that go straight across the $Y$ box, as seen in the intermediate box of Figure~\ref{fig:unsafe}{\bf b}. 
\end{rem}

\begin{figure}[!ht]
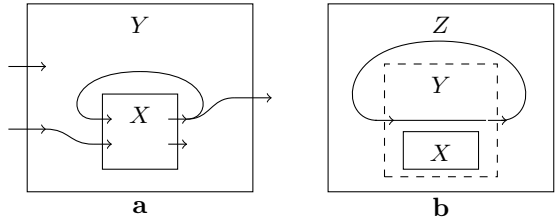

\activetikz{
	\node at (1.5,-.2){\bf a};
	\path(0,0);\blackbox{(3,2.5)}{2}{1}{\small$Y$}{.5}
	\path(1,.3);\blackbox{(1,1)}{2}{2}{\small$X$}{.25}
	\fancyarc{(.875,.9667)}{(2.125,.9667)}{10}{18}
	\directarc{(2.125,.9667)}{(2.75,1.25)}
	\directarc{(.25,.8333)}{(.875,.6333)}
	\node at (5.5,-.2){\bf b};
	\path(4,0);\blackbox{(3,2.5)}{0}{0}{\small$Z$}{.5}
	\path(4.75,.2);\dashbox{(1.5,1.5)}{1}{1}{\small$Y$}{.25}
	\path(5,.3); \blackbox{(1,.5)}{0}{0}{\small$X$}{.1}
	\directarc{(4.85,.95)}{(6.1,.95)}
	\fancyarc{(4.65,.95)}{(6.3,.95)}{15}{30}
}
\caption{{\bf (a)} A faux-wiring diagram violating the bijectivity condition in Definition \ref{def:W}. \\ {\bf (b)} A composition of diagrams in which a loop emerges because the inner diagram has a (prohibited) passing wire.}
\label{fig:unsafe}
\end{figure}

Now that we have formally defined and concretely explicated the category $\bfW$, we will make it into a monoidal category by defining its tensor product.

\begin{defn}\label{def:mon}
Let $X_1,X_2,,Y_1,Y_2\in\Ob\bfW$ be boxes and $\Phi_1\taking X_1\to Y_2$ and $\Phi_2\taking X_2\to Y_2$ be wiring diagrams. The \emph{monoidal product} $\oplus$ is given by
\[X_1\oplus X_2:=\left(\inp{X}_1+\inp{X}_2\;,\;\outp{X}_1+\outp{X}_2\;\right), \hspace{10 mm} \Phi_1\oplus\Phi_2:=\Phi_1+\Phi_2.\]
The \emph{closed box} $0=\{\varnothing,\varnothing\}$ is the monoidal unit.
\end{defn}

\begin{rem}
Once we add semantics in Section~\ref{sec:g}, closed boxes will correspond to \emph{autonomous systems}, which do not interact with any outside environment (see Remark~\ref{ex:dynamical system}).
\end{rem}

We now make this monoidal product explicit with an example.

\begin{ex} Consider boxes $X=(\{x_1,x_2\},\{x_3,x_4\})$ and $Y=(\{y_1\},\{y_2,y_3\})$ depicted below.

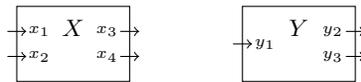
\activetikz{
	\path(-1.5,0);\blackbox{(1.5,1)}{2}{2}{\small$X$}{.25}
	\node at (-1.2,.666) {\tiny$x_1$};
	\node at (-1.2,.333) {\tiny$x_2$};
	\node at (-.3,.666) {\tiny$x_3$};
	\node at (-.3,.333) {\tiny$x_4$};
	\path(1.5,0);\blackbox{(1.5,1)}{1}{2}{\small$Y$}{.25}
	\node at (1.8,.5){\tiny$y_1$};
	\node at (2.7,.666){\tiny$y_2$};
	\node at (2.7,.333){\tiny$y_3$};
} We depict their tensor $X\oplus Y=(\{x_1,x_2,y_1\},\{x_3,x_4,y_2,y_3\})$ by stacking boxes.

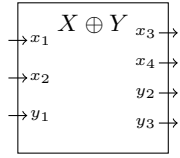
\activetikz{
	\path(0,0);\blackbox{(2,2)}{3}{4}{\small$X\oplus Y$}{.25}
	\node at (.3,1.5) {\tiny$x_1$};
	\node at (.3,1) {\tiny$x_2$};
	\node at (.3,.5) {\tiny$y_1$};
	\node at (1.7,1.6) {\tiny$x_3$};
	\node at (1.7,1.2){\tiny$x_4$};
	\node at (1.7,.8){\tiny$y_2$};
	\node at (1.7,.4){\tiny$y_3$};
}

Similarly, consider the following wiring diagrams (with ports left unlabelled).

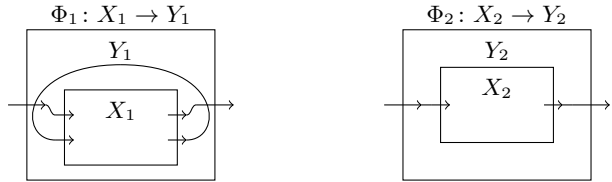
\activetikz{
	\node at (1.25,2.2){\small $\Phi_1 \taking X_1\to Y_1$};
	\path(0,0);\blackbox{(2.5,2)}{1}{1}{\small$Y_1$}{.5}
	\path(.5,.2);\blackbox{(1.5,1)}{2}{2}{\small $X_1$}{.25}
	\directarc{(.25,1)}{(.375,0.8666)}
	\directarc{(2.125,0.8666)}{(2.25,1)}
	\fancyarc{(.375,0.5333)}{(2.125,0.5333)}{14}{28.5}
	\node at (6.25,2.2){\small $\Phi_2 \taking X_2\to Y_2$};
	\path(5,0);\blackbox{(2.5,2)}{1}{1}{\small $Y_2$}{.5}
	\path(5.5,.5);\blackbox{(1.5,1)}{1}{1}{\small $X_2$}{.25}
    \directarc{(5.25,1)}{(5.375,1)}
	\directarc{(7.125,1)}{(7.25,1)}
} We can depict their composition via stacking.

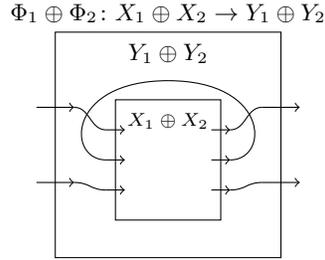
\activetikz{
	\node at (1.5,3.25){\small $\Phi_1\oplus\Phi_2 \taking X_1\oplus X_2\to Y_1\oplus Y_2$};
	\path(0,0);\blackbox{(3,3)}{2}{2}{\small$Y_1\oplus Y_2$}{.5}
	\path(.8,.5);\blackbox{(1.4,1.6)}{3}{3}{\scriptsize$X_1\oplus X_2$}{.25}
	\directarc{(.25,2)}{(.675,1.7)}
	\directarc{(2.325,1.7)}{(2.75,2)}
	\fancyarc{(.675,1.3)}{(2.325,1.3)}{15}{30}
	\directarc{(.25,1)}{(.675,.9)}
	\directarc{(2.325,.9)}{(2.75,1)}
}
\end{ex}


We now prove that the above data characterizing $(\bfW,\oplus,0)$ indeed constitutes a symmetric monoidal category, at which point we can, as advertised, invoke Definition~\ref{def:SMC to Opd} to define the operad $\Opd{\bfW}$.

\begin{prop} \label{prop:W is SMC}
The category $\bfW$ in Definition~\ref{def:W} and the monoidal product $\oplus$ with unit $0$ in Definition~\ref{def:mon} form a symmetric monoidal category $(\bfW,\oplus,0)$.
\end{prop}

\begin{proof}
We begin by establishing that $\bfW$ is indeed a category. We first show that our class of wiring diagrams is closed under composition. Let $\Phi=(X,Y,\varphi)$, $\Psi=(Y,Z,\psi)$, and $\Psi\circ\Phi=(X,Z,\omega)$.

To show that $\omega$ is a typed bijection, we replace the pair of maps $(\inp{\varphi},\outp{\varphi})$ with a pair of bijections $(\widetilde{\inp{\varphi}},\widetilde{\outp{\varphi}})$ as follows. Let $\expt{X}{\varphi}\ss\outp{X}$ (for \emph{exports}) denote the image of $\outp{\varphi}$, and $\loc{X}{\varphi}$ (for \emph{local ports}) be its complement. Then we can identify $\varphi$ with the following pair of typed bijections
\begin{displaymath}
   \left\{
     \begin{array}{lr}
       \widetilde{\inp{\varphi}}\taking \inp{X} \xrightarrow{\cong} \loc{X}{\varphi}+\inp{Y} \\
       \widetilde{\outp{\varphi}}\taking \outp{Y} \xrightarrow{\cong} \expt{X}{\varphi}
     \end{array}
   \right.
\end{displaymath}

 Similarly, identify $\psi$ with $(\widetilde{\inp{\psi}},\widetilde{\outp{\psi}})$. We can then rewrite the diagram defining $\omega$ in (\ref{dia:composition diagrams}) as one single commutative diagram of typed finite sets.
\[
\xymatrixcolsep{4pc}
\xymatrix{
\inp{X}+\outp{Z}
    \ar[d]_{\widetilde{\inp{\varphi}}+\widetilde{\outp{\psi}}} 
    \ar@{-->}[r]^{\omega}
&\outp{X}+\inp{Z} \\
\loc{X}{\varphi}+\inp{Y}+\expt{Y}{\psi}
    \ar[d]_{\bigid _{\loc{X}{\varphi}}+\widetilde{\inp{\psi}}+\bigid _{\expt{Y}{\psi}}}
&\loc{X}{\varphi}+\expt{X}{\varphi}+\inp{Z}
    \ar[u]_{\cong} \\
\loc{X}{\varphi}+\loc{Y}{\psi}+\inp{Z}+\expt{Y}{\psi}
    \ar[r]_-{\cong} & 
\loc{X}{\varphi}+\outp{Y}+\inp{Z}
    \ar[u]_{\bigid _{\loc{X}{\varphi}}+\widetilde{\outp{\varphi}}+\bigid _{\inp{Z}}}
}\]
As a composition of typed bijections, $\omega$ is also a typed bijection.

The following computation proves that $\omega$ has no passing wires:

\[\omega(\outp{Z})=\varphi\big(\psi(\outp{Z})\big)\subseteq\varphi(\outp{Y})\subseteq \outp{X}.\] 

Therefore $\bfW$ is closed under wiring diagram composition. To show that $\bfW$ is a category, it remains to prove that composition of wiring diagrams satisfies the unit and associativity axioms. The former is straightforward and will be omitted. We now establish the latter.

Consider the wiring diagrams $\Theta=(V,X,\theta),\Phi=(X,Y,\varphi),\Psi=(Y,Z,\psi)$; and let  $(\Psi\circ\Phi)\circ\Theta=(V,Z,\kappa)$ and $\Psi\circ(\Phi\circ\Theta)= (V,Z,\lambda)$. We readily see that $\outp{\kappa}=\outp{\lambda}$ by the associativity of composition in $\TFS{}$. Proving that $\inp{\kappa}=\inp{\lambda}$ is equivalent to establishing the commutativity of the following diagram:

\begin{equation} \label{eqn:associativity}
\xymatrix@C=31pt@R=1.8pc{
{} &\outp{V}+\inp{Z} &{} \\
{} &\outp{V}+\outp{V}+\inp{Z} \ar[u]^{\nabla+\bigid } &{} \\
\outp{V}+\outp{Y}+\inp{Z}  \ar[r]^-{\bigid +\outp{\varphi}+\bigid }
&\outp{V}+\outp{X}+\inp{Z}  \ar[u]^{\bigid + \outp{\theta}+\bigid } 
&\outp{V}+\outp{X}+\outp{X}+\inp{Z} \ar[l]_-{\bigid +\nabla+\bigid } \\
\outp{V}+\inp{Y} \ar[u]^{\bigid +\inp{\psi}} &{} &{} \\
\outp{V}+\outp{V}+\inp{Y} \ar[u]^{\nabla+\bigid } 
&\outp{V}+\outp{X}+\inp{Y} \ar[l]^-{\bigid + \outp{\theta}+\bigid } \ar[r]_-{\bigid +\bigid +\inp{\psi}}
&\outp{V}+\outp{X}+\outp{Y}+\inp{Z} \ar[uu]_{\bigid +\bigid +\outp{\varphi}+\bigid } \\
{} &\outp{V}+\inp{X} \ar[u]^{\bigid + \inp{\varphi}} &{} \\
{} &\inp{V} \ar[u]^{\inp{\theta}}&{}
}
\end{equation}

This diagram commutes in any category with coproducts, as follows from
the associativity and naturality of the codiagonal map. We present a formal argument of this fact below in the language
of string diagrams (See \cite{JSTensor}). As in \cite{Selinger}, we
let squares with blackened corners denote generic morphisms. We let
triangles denote codiagonal maps.  See Figure~\ref{string} below.

\begin{figure}[ht]\label{fig:string} 
\centering
\begin{tikzpicture} [scale=.8]
\node[mysquare] at (2,0) (Theta11) {$\theta^\text{out}$};
\node[fold] at (4.5,.35) (V1) {};
\node[mysquare] at (3.5,-1.15) (Psi) {$\psi^\text{in}$};
\node[mysquare] at (6,-.7) (Phi) {$\varphi^\text{out}$};
\node[mysquare] at (8.5,-.7) (Theta12) {$\theta^\text{out}$};
\node[fold] at (10.7,-.13) (V2) {};

\draw[onearrow={0.4}{$X^\text{out}$}] (.13,0) -- (Theta11.west);
\draw[onearrow={0.5}{$V^\text{out}$}] (Theta11.east) -- ([yshift=-10pt]V1.west);
\draw[onearrow={0.15}{$V^\text{out}$}] (.13,.75) -- ([yshift=11pt]V1.west);
\draw[onearrow={0.5}{$Y^\text{out}$}] ([yshift=12pt]Psi.east) -- ([yshift=-.5pt]Phi.west);
\draw[onearrow={0.96}{$Z^\text{in}$}] ([yshift=-11pt]Psi.east) -- (13,-1.5);
\draw[onearrow={0.5}{$X^\text{out}$}] (Phi.east) -- (Theta12.west);
\draw[onearrow={0.2}{$Y^\text{in}$}] (.13,-1.15) -- (Psi.west);
\draw[onearrow={0.5}{$V^\text{out}$}] (V1.east) -- ([yshift=12pt]V2.west);
\draw[onearrow={0.5}{$V^\text{out}$}] ([yshift=6pt]Theta12.east) -- ([yshift=-10pt]V2.west);
\draw[onearrow={0.7}{$V^\text{out}$}] (V2.east) -- (13,-.13);
\end{tikzpicture}
\begin{tikzpicture} [scale=.8]
\node[mysquare] at (6.5,0.5) (Theta11) {$\theta^\text{out}$};
\node[mysquare] at (2,-1.15) (Psi) {$\psi^\text{in}$};
\node[mysquare] at (4.2,-.7) (Phi) {$\varphi^\text{out}$};
\node[mysquare] at (6.5,-.7) (Theta12) {$\theta^\text{out}$};
\node[fold] at (8.7,.8) (V1) {};
\node[fold] at (11.1,.15) (V2) {};

\draw[onearrow={0.5}{$V^\text{out}$}] ([yshift=-1.7pt]Theta11.east) -- ([yshift=-10pt]V1.west);
\draw[onearrow={0.5}{$Y^\text{out}$}] ([yshift=12pt]Psi.east) -- ([yshift=-.5pt]Phi.west);
\draw[onearrow={0.96}{$Z^\text{in}$}] ([yshift=-11pt]Psi.east) -- (13,-1.5);
\draw[onearrow={0.5}{$X^\text{out}$}] (Phi.east) -- (Theta12.west);
\draw[onearrow={.35}{$Y^\text{in}$}] (0,-1.15) -- (Psi.west);
\draw[onearrow={0.088}{$X^\text{out}$}] (0,.5) -- (Theta11.west);
\draw[onearrow={0.5}{$V^\text{out}$}] ([yshift=12pt]Theta12.east) -- ([yshift=-12pt]V2.west);
\draw[onearrow={0.5}{$V^\text{out}$},rounded corners=20pt] (V1.east) -- ([yshift=12pt]V2.west);
\draw[onearrow={0.065}{$V^\text{out}$}]
(0,1.2) -- ([yshift=12pt]V1.west);
\draw[onearrow={0.5}{$V^\text{out}$}] (V2.east) -- (13,.15);
\end{tikzpicture}
\begin{tikzpicture} [scale=.8]
\node[mysquare] at (6.5,0.5) (Theta11) {$\theta^\text{out}$};
\node[mysquare] at (2,-1.15) (Psi) {$\psi^\text{in}$};
\node[mysquare] at (4.2,-.7) (Phi) {$\varphi^\text{out}$};
\node[mysquare] at (6.5,-.7) (Theta12) {$\theta^\text{out}$};
\node[fold] at (8.7,-.13) (V1) {};
\node[fold] at (11.1,.3) (V2) {};

\draw[onearrow={0.5}{$V^\text{out}$}] ([yshift=-6.5pt]Theta11.east) -- ([yshift=11pt]V1.west);
\draw[onearrow={0.5}{$Y^\text{out}$}] ([yshift=12pt]Psi.east) -- ([yshift=-.5pt]Phi.west);
\draw[onearrow={0.96}{$Z^\text{in}$}] ([yshift=-11pt]Psi.east) -- (13,-1.5);
\draw[onearrow={0.5}{$X^\text{out}$}] (Phi.east) -- (Theta12.west);
\draw[onearrow={.35}{$Y^\text{in}$}] (0,-1.15) -- (Psi.west);
\draw[onearrow={0.088}{$X^\text{out}$}] (0,.5) -- (Theta11.west);
\draw[onearrow={0.5}{$V^\text{out}$}] ([yshift=6pt]Theta12.east) -- ([yshift=-10pt]V1.west);
\draw[onearrow={0.5}{$V^\text{out}$}] (V1.east) -- ([yshift=-12pt]V2.west);
\draw[onearrow={0.05}{$V^\text{out}$},rounded corners=20pt]
(0,1.2) -- (8,1.2) -- ([yshift=12pt]V2.west);
\draw[onearrow={0.5}{$V^\text{out}$}] (V2.east) -- (13,.3);
\end{tikzpicture}
\begin{tikzpicture} [scale=.8]
\node[mysquare] at (9.2,-0.3) (Theta) {$\theta^\text{out}$};
\node[mysquare] at (2,-1.15) (Psi) {$\psi^\text{in}$};
\node[mysquare] at (4.2,-.7) (Phi) {$\varphi^\text{out}$};
\node[fold] at (6.5,-.3) (V1) {};
\node[fold] at (11.1,.3) (V2) {};

\draw[onearrow={0.5}{$Y^\text{out}$}] ([yshift=12pt]Psi.east) -- ([yshift=-.6pt]Phi.west);
\draw[onearrow={0.96}{$Z^\text{in}$}] ([yshift=-11pt]Psi.east) -- (13,-1.5);
\draw[onearrow={0.55}{$X^\text{out}$}] (Phi.east) -- ([yshift=-11.pt]V1.west);
\draw[onearrow={.35}{$Y^\text{in}$}] (0,-1.15) -- (Psi.west);
\draw[onearrow={0.088}{$X^\text{out}$}] (0,.1) -- ([yshift=12pt]V1.west);
\draw[onearrow={0.5}{$X^\text{out}$}] (V1.east) -- (Theta.west);
\draw[onearrow={0.5}{$V^\text{out}$}] ([yshift=6pt]Theta.east) -- ([yshift=-10pt]V2.west);
\draw[onearrow={0.05}{$V^\text{out}$}]
(0,.75) -- ([yshift=12pt]V2.west);
\draw[onearrow={0.5}{$V^\text{out}$}] (V2.east) -- (13,.3);
\end{tikzpicture}
\caption{String diagram proof of commutativity of \eqref{eqn:associativity} }
\label{string}
\end{figure}
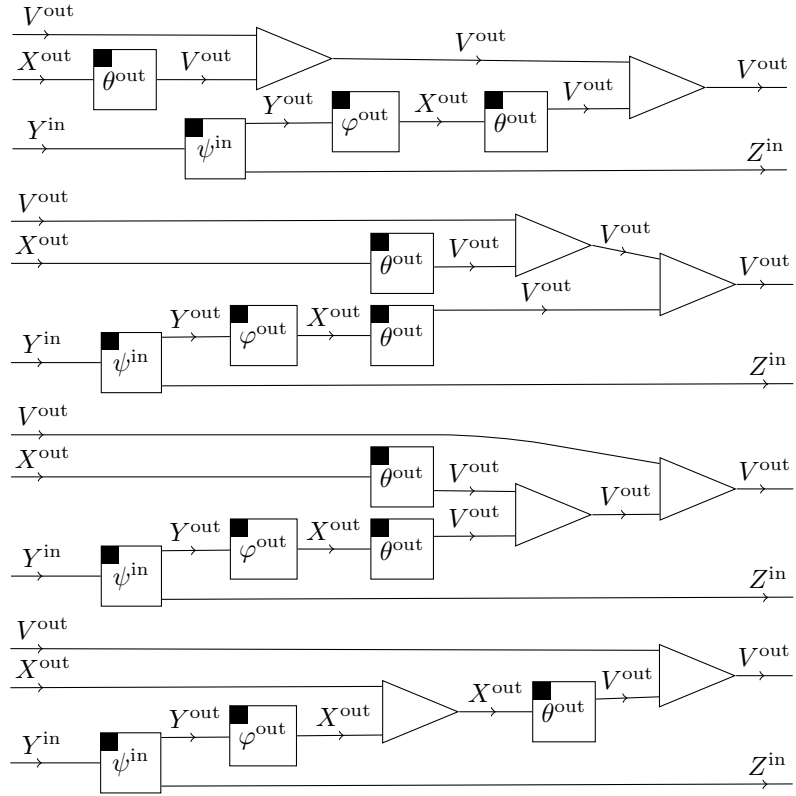

The first step of the proof follows from the topological nature of string diagrams, which mirror the axioms of monoidal categories. The second step invokes the associativity of codiagonal maps. The third and final step follows from the naturality of codiagonal maps, i.e., the commutativity of the following diagram.

\[\xymatrix{\outp{V}+\outp{V} \ar[r]^-{\nabla} \ar[d]_{\outp{\theta}+\outp{\theta}} & \outp{V} \ar[d]^{\outp{\theta}} \\
\outp{X}+\outp{X} \ar[r]^-{\nabla} & \outp{X}
}\]

Now that we have shown that $\bfW$ is a category, we show that $(\oplus,0)$ is a monoidal structure on $\bfW$. Let $X,X',X''\in\Ob\bfW$ be boxes. We readily observe the following canonical isomorphisms. 
\begin{align*}
&X\oplus 0= X= 0\oplus X &\emph{(unity)}\\
&(X\oplus X')\oplus X''= X\oplus (X'\oplus X'')&\emph{(associativity)}\\
&X\oplus X'= X'\oplus X &\emph{(commutativity)}
\end{align*}
Hence the monoidal product $\oplus$ is well behaved on objects. It is similarly easy, and hence will be omitted, to show that $\oplus$ is functorial. This completes the proof that $(\bfW,\oplus,0)$ is a symmetric monoidal category.
\end{proof}

Having established that $(\bfW,\oplus,0)$ is an SMC, we can now speak about the operad $\Opd{\bfW}$ of wiring diagrams. In particular, we can draw operadic pictures, such as the one in our motivating example in Figure~\ref{fig:pipebrine}, to which we now return.

\begin{ex}\label{ex:wiring explained}
Figure~\ref{fig:WD} depicts an $\Opd{\bfW}$ wiring diagram $\Phi\taking X_1,X_2\to Y$, which we may formally denote by the tuple $\Phi=(X_1,X_2;Y;\varphi)$. Reading directly from Figure~\ref{fig:WD}, we have the boxes:

\begin{align*}
X_1&=\big(\{\inp{X}_{1a},\inp{X}_{1b}\},\{\outp{X}_{1a}\}\big) \\
X_2&=\big(\{\inp{X}_{2a},\inp{X}_{2b}\} ,\{\outp{X}_{2a},\outp{X}_{2b}\}\big) \\
Y&=\big(\{\inp{Y}_a,\inp{Y}_b\},\{\outp{Y}_a\}\big)
\end{align*}

The wiring diagram $\Phi$ is visualized by nesting the domain boxes $X_1,X_2$ within the codomain box $Y$, and drawing the wires prescribed by $\varphi$, as recorded below in Table~\ref{tab:explicit}.

\noindent\begin{minipage}{\linewidth}
\[
\begin{array}{c||c|c|c|c|c}
\rule[-4pt]{0pt}{16pt}
w\in\inp{X}+\outp{Y}&\inp{X}_{1a}&\inp{X}_{1b}&\inp{X}_{2a}&\inp{X}_{2b}&\outp{Y}_{a}
\\\hline
\rule[-4pt]{0pt}{16pt}
\varphi(w)\in\outp{X}+\inp{Y}&\inp{Y}_{b}&\outp{X}_{2b}&\inp{Y}_{a}&\outp{X}_{1a}&\outp{X}_{2a}
\end{array}
\]
\smallskip
\captionof{table}{} \label{tab:explicit} 
\end{minipage}

\vspace{-3 mm}

\begin{figure}[ht]
\activetikz{
	\path(0,0);
	\blackbox{(10,5)}{2}{1}{$Y$}{.7}
	        \node at (.4,3.6) {\small $\inp{Y}_{a}$};
	        \node at (.4,1.9) {\small $\inp{Y}_{b}$};
	        \node at (9.6,2.75) {\small $\outp{Y}_a$};
	\path(2,1.5);
	\blackbox{(2,2)}{2}{1}{$X_1$}{.5}
	        \node at (1.7,3.06) {\small $\inp{X}_{1a}$};
	        \node at (1.7,2.4) {\small $\inp{X}_{1b}$};
	        \node at (4.38,2.7) {\small $\outp{X}_{1a}$};
	\path(6,1.5);
	\blackbox{(2,2)}{2}{2}{$X_2$}{.5};
	        \node at (5.7,3.07) {\small $\inp{X}_{2a}$};
	        \node at (5.7,2.4) {\small $\inp{X}_{2b}$};
	        \node at (8.37,3.07) {\small $\outp{X}_{2a}$};
	        \node at (8.37,2.37) {\small $\outp{X}_{2b}$};
	\directarc{(4.25,2.5)}{(5.75,2.16667)} 
	    \node at (5,2) {};
	\directarc{(0.35,1.6667)}{(1.75,2.83333)} 
	    \node at (.7,1.4) {};
	    \node at (.7,1.2) {};
	\fancyarc{(0.35,3.3333)}{(5.75,2.83333)}{-40}{25} 
	    \node at (.6,3.1) {};
	    \node at (.6,2.9) {};
	\directarc{(8.25,2.8333)}{(9.65,2.5)} 
	    \node at (9.5,2.3) {};
	    \node at (9.5,2.1) {};
	\fancyarc{(1.75,2.16667)}{(8.25,2.16667)}{20}{-45} 
}
\caption{A wiring diagram $\Phi\taking X_1,X_2\to Y$ in $\Opd{\bfW}$.}
\label{fig:WD}
\end{figure}

To reconceptualize $\Phi\taking X_1,X_2\to Y$ as a wiring diagram in $\bfW$, we simply consider the tensor $\Phi\taking X_1\oplus X_2\to Y$, as given in Figure~\ref{fig:reconcept} below. This demonstrates the fact that operadic pictures are easier to read and hence are more illuminating.

\begin{figure}[ht]
\activetikz{
	\path(0,0);
	\blackbox{(7,8)}{2}{1}{$Y$}{.7}
	    \node at (0.35,5.666) {\small $\inp{Y_a}$};
	    \node at (0.35,3) {\small$\inp{Y_b}$};
	    \node at (6.6,4.233) {\small$\outp{Y_a}$};
	\path(2,1.5);
	\blackbox{(3,5)}{4}{3}{\small$X_1\oplus X_2$}{.5}
        \node at (1.7,5.8) {\small$\inp{X_{1a}}$};
	    \node at (1.7,4.8) {\small$\inp{X_{1b}}$};
	    \node at (1.7,3.8) {\small$\inp{X_{2a}}$};
	    \node at (1.7,2.8) {\small$\inp{X_{2b}}$};
	    \node at (5.4,5.55) {\small$\outp{X_{1a}}$};
	    \node at (5.4,4.3) {\small$\outp{X_{2a}}$};
	    \node at (5.4,3.05) {\small$\outp{X_{2b}}$};
    \directarc{(0.35,5.333)}{(1.75,3.5)}
    \directarc{(0.35,2.666)}{(1.75,5.5)}
    \directarc{(5.25,4)}{(6.65,4)}
    \fancyarc{(1.75,4.5)}{(5.25,2.75)}{40}{-80}
    \fancyarc{(1.75,2.5)}{(5.25,5.25)}{35}{95}
}
\caption{A wiring diagram $\Phi\taking X_1\oplus X_2\to Y$ in $\bfW$ corresponding to the $\Opd{\bfW}$ wiring diagram $\Phi:X_1,X_2\to Y$ of Figure~\ref{fig:WD}.}
\label{fig:reconcept}
\end{figure}
\end{ex}

The following remark explains that our pictures of wiring diagrams are not completely ad hoc---they are depictions of 1-dimensional oriented manifolds with boundary. The boxes in our diagrams simply tie together the positively and negatively oriented components of an individual oriented 0-manifold.

\begin{rem}\label{rem:cobordism} For any set $S$, let $\operatorname{1--\bf Cob}/S$ denote the symmetric monoidal category of oriented 0-manifolds over $S$ and the 1-dimensional cobordisms between them. We call its objects \emph{oriented $S$-typed 0-manifolds}. Recall that $\bfW=\bfW_{\Man}$ is our category of $\Man$-typed wiring diagrams; let ${\mathbf M}:=\Ob\Man$ denote the set of manifolds (see Remark~\ref{rem:default}). There is a faithful,  essentially surjective, strong monoidal functor 
\[\bfW\to \operatorname{1--\bf Cob}/{\mathbf M},\] 
sending a box $(\inp{X},\outp{X})$ to the oriented ${\mathbf M}$-typed 0-manifold $\inp{X}+\outp{X}$ where $\inp{X}$ is oriented positively and $\outp{X}$ negatively. Under this functor, a wiring diagram $\Phi=(X,Y,\varphi)$ is sent to a 1-dimensional cobordism that has no closed loops. A connected component of such a cobordism can be identified with either its left or right endpoint, which correspond to the domain or codomain of the bijection \mbox{$\varphi\taking\inp{X}+\outp{Y}\To{\iso}\outp{X}+\inp{Y}$}. See \cite{SpivakSchultzRupel}.

In fact, with the {\em no passing wires} condition on morphisms (cobordisms) $X\to Y$ (see Definition \ref{def:W}), the subcategory $\bfW\ss\operatorname{1--\bf Cob}/{\mathbf M}$ is the left class of an orthogonal factorization system. See \cite{Abadi}.
\end{rem}

Let $\Phi=(X,Y,\varphi)$ be a wiring diagram. Applying the dependent product functor (see Definition~\ref{def:depprod}) to $\varphi$, we obtain
a diffeomorphism of manifolds 
\begin{equation}\label{eqn:prodwd}\overline{\varphi}\taking \overline{\outp{X}}\times\overline{\inp{Y}}\to \overline{\inp{X}}\times\overline{\outp{Y}}.\end{equation}
Equivalently, if $\varphi$ is represented by the pair $(\inp{\varphi},\outp{\varphi})$, as in Definition~\ref{def:W}, we can express $\ol{\varphi}$ in terms of its pair of component maps:
\begin{displaymath}
   \left\{
     \begin{array}{lr}
       \overline{\inp{\varphi}}\taking  \overline{\outp{X}}\times\overline{\inp{Y}}\to\overline{\inp{X}} \\
       \overline{\outp{\varphi}}\taking \overline{\outp{X}}\to\overline{\outp{Y}}
     \end{array}
   \right.
\end{displaymath}

It will also be useful to apply the dependent product functor to the commutative diagrams in (\ref{dia:composition diagrams}), which define wiring diagram composition. Note that, by the contravariance of the dependent product, the codiagonal $\nabla\taking \outp{X}+\outp{X}\to \outp{X}$ gets sent to the diagonal map $\Delta\taking \overline{\outp{X}}\to\overline{\outp{X}}\times\overline{\outp{X}}$. Thus we have the following commutative diagrams:
\begin{align}\label{dia:dep prod of wd}
\xymatrixcolsep{3.5pc}
\xymatrix{
\overline{\outp{X}}\times\overline{\inp{Z}}
    \ar[r]^{\overline{\inp{\omega}}}
    \ar[d]_{\Delta\times\bigid }
&\overline{\inp{X}} 
\\
\overline{\outp{X}}\times\overline{\outp{X}}\times\overline{\inp{Z}}
    \ar[d]_{\bigid \times\overline{\outp{\varphi}}\times\bigid }
&{}    
\\
\overline{\outp{X}}\times\overline{\outp{Y}}\times\overline{\inp{Z}}
    \ar[r]_-{\bigid \times\overline{\inp{\psi}}}
&\overline{\outp{X}}\times\overline{\inp{Y}}
    \ar[uu]_{\overline{\inp{\varphi}}}
} \mskip15mu
\xymatrixcolsep{1.5pc}
\xymatrix{
\overline{\outp{X}}
    \ar[rd]_{\overline{\outp{\varphi}}} 
    \ar[rr]^{\overline{\outp{\omega}}}
&
&\overline{\outp{Z}}\\
&\overline{\outp{Y}} 
    \ar[ru]_{\overline{\outp{\psi}}}}
\end{align}

\section{The Algebra of Open Systems}
\label{sec:g}

In this section we define an algebra $\mcG\taking(\bfW,\oplus,0)\to(\mathbf{Set},\times,\star)$ (see Definition~\ref{def:algebra}) of general open dynamical systems. A $\bfW$-algebra can be thought of as a choice of semantics for the syntax of $\bfW$, i.e., a set of possible meanings for boxes and wiring diagrams. As in Definition~\ref{def:SMC to Opd}, we may use this to construct the corresponding operad algebra $\Opd{\mcG }:\Opd{\bfW}\to\mathbf{Sets}$. Before we define $\mcG$, we revisit Example~\ref{ex:main} for inspiration.

\begin{ex} \label{ex:promise} As the textbook exercise \cite[Problem 7.21]{BD} prompts, let's begin by writing down the system of equations that governs the amount of salt $Q_i$ within the tanks $X_i$. This can be done by using dimensional analysis for each port of $X_i$ to find the the rate of salt being carried in ounces per minute, and then equating the rate $\dot{Q}_i$ to the sum across these rates for $\inp{X}_i$ ports minus $\outp{X}_i$ ports.

\begin{align*}
\dot{Q}_1\frac{\text{oz}}{\text{min}}&=
-\left(\frac{Q_1 \text{oz}}{30 \text{gal}}\cdot\frac{3 \text{gal}}{\text{min}}\right)
+\left(\frac{Q_2\text{oz}}{20\text{gal}}\cdot\frac{1.5\text{gal}}{\text{min}}\right)
+\left(\frac{1\text{oz}}{\text{gal}}\cdot\frac{1.5\text{gal}}{\text{min}}\right) \\
\dot{Q}_2\frac{\text{oz}}{\text{min}}&=
-\left(\frac{Q_2 \text{oz}}{20 \text{gal}}\cdot\frac{(1.5+2.5) \text{gal}}{\text{min}}\right)
+\left(\frac{Q_1\text{oz}}{30\text{gal}}\cdot\frac{3\text{gal}}{\text{min}}\right)
+\left(\frac{3\text{oz}}{\text{gal}}\cdot\frac{1\text{gal}}{\text{min}}\right)
\end{align*}

Dropping the physical units, we are left with the following system of ODEs:

\begin{equation}\label{eqn:naive}
   \left\{
     \begin{array}{lr}
       \dot{Q}_1=-.1Q_1+.075Q_2+1.5 \\
       \dot{Q}_2=.1Q_1-.2Q_2+3
     \end{array}
   \right.
\end{equation}
\end{ex}

The derivations for the equations in (\ref{eqn:naive}) involved a hidden step in which the connection pattern in Figure~\ref{fig:pipebrine}, or equivalently Figure~\ref{fig:WD}, was used. Our wiring diagram approach explains this step and makes it explicit. Each box in a wiring diagram should only ``know'' about its own inputs and outputs, and not how they are connected to others. That is, we can only define a system on $X_i$ by expressing $\dot{Q}_i$ just in terms of $Q_i$ and $\inp{X}_i$---this is precisely the data of an open system (see Definition~\ref{def:opensystem}). We now define our algebra $\mcG$, which assigns a set of open systems to a box. Given a wiring diagram and an open system on its domain box, it also gives a functorial procedure for assigning an open system to the codomain box. We will then use this new machinery to further revisit Example~\ref{ex:promise} in Example~\ref{ex:as promised}. 

\begin{defn}\label{def:general algebra} We define $\mcG:(\bfW,\oplus,0)\to(\Set,\times,\star)$ as follows. Let $X\in\Ob\bfW$. The \emph{set of open systems on $X$}, denoted $\mcG(X)$, is defined as
\[\mcG (X)=\{(S,f)\; |\;S\in\Ob\TFS{},(\overline{S},\overline{\inp{X}},\overline{\outp{X}},f)\in\Ob\ODS \}.\] 
We call $S$ the set of \emph{state variables} and its dependent product $\overline{S}$ the \emph{state space}.

Let $\Phi=(X,Y,\varphi)$ be a wiring diagram. Then $\mcG (\Phi)\taking \mcG (X)\to\mcG (Y)$ is given by $(S,f)\mapsto (\mcG (\Phi)S,\mcG (\Phi)f)$, where $\mcG (\Phi)S=S$ and $g=\mcG (\Phi)f\taking \overline{S}\times\overline{\inp{Y}}\to T\overline{S}\times\overline{\outp{Y}}$ is defined by the dashed arrows $(\inp{g},\outp{g})$ (see Definition~\ref{def:opensystem}) that make the diagrams below commute:
\begin{equation}
\xymatrixcolsep{3.5pc}
\xymatrix{
\overline{S}\times\overline{\inp{Y}}
    \ar[d]_{\Delta\times\bigid _{\overline{\inp{Y}}}} 
    \ar@{-->}[r]^-{\inp{g}} 
&T\overline{S}  
\\ 
\overline{S}\times\overline{S}\times\overline{\inp{Y}}
    \ar[d]_{\bigid _{\overline{S}}\times \outp{f}\times\bigid _{\overline{\inp{Y}}}}
& {}
\\
\overline{S}\times\overline{\outp{X}}\times\overline{\inp{Y}}
    \ar[r]_-{\bigid _{\overline{S}}\times\overline{\inp{\varphi}}}
&\overline{S}\times\overline{\inp{X}}
    \ar[uu]_{\inp{f}}
} \hspace{7 mm} \xymatrixcolsep{2.5pc}
\xymatrix{
\overline{S} 
    \ar[rd]_{\outp{f}} 
    \ar@{-->}[rr]^{\outp{g}}
&
&\overline{\outp{Y}}\\
&\overline{\outp{X}} 
    \ar[ru]_{\overline{\outp{\varphi}}}}
\label{eqn:G of wd}
\end{equation} One may note strong resemblance between the diagrams in (\ref{eqn:G of wd}) and those in (\ref{dia:composition diagrams}).

We give $\mcG $ a lax monoidal structure: for any pair $X,X'\in\bfW$ we have a coherence map  $\mu_{X,X'}:\mcG (X)\times\mcG (X')\to\mcG (X\oplus X')$ given by \[\big((S,f),(S',f')\big)\mapsto (S+S',f\times f'),\] where $f\times f'$ is as in Lemma~\ref{def:osprod}.
\label{def:mu}
\end{defn}

\begin{rem} Recall from Remark \ref{rem:default} that $\Man$ is small, so the collection $\mcG(X)$ of open systems on $X$ is indeed a set.
\end{rem}

%

\begin{rem} One may also encode an initial condition in $\mcG $ by using $\Man_*$ instead of $\Man$ in Remark~\ref{rem:default} as the default choice of finite product category, where $\Man_*$ is the category of pointed smooth manifolds and base point preserving smooth maps. The base point represents the initialization of the state variables.
\end{rem}


We now establish that $\mcG$ is indeed an algebra.

\begin{prop}\label{prop:G is W-alg}
The pair $(\mcG,\mu)$ of Definition~\ref{def:general algebra} is a lax monoidal functor, i.e., $\mcG$ is a $\bfW$-algebra.
\end{prop}

\begin{proof}
Let $\Phi=(X,Y,\varphi)$ and $\Psi=(Y,Z,\psi)$ be wiring diagrams in $\bfW$. To show that $\mcG$ is a functor, we must have that $\mcG (\Psi\circ\Phi)=\mcG (\Psi)\circ\mcG (\Phi)$. Immediately we have $\mcG (\Psi\circ\Phi)S=S=\mcG (\Psi)(\mcG (\Phi)S)$. 

Now let \mbox{$h:=\mcG (\Psi\circ\Phi)f$} and $k:=\mcG (\Psi)(\mcG (\Phi)f)$. It suffices to show $h=k$, or equivalently $(\inp{h},\outp{h})=(\inp{k},\outp{k})$. One readily sees that $\outp{h}=\outp{k}$. We use (\ref{dia:dep prod of wd}) and (\ref{eqn:G of wd}) to produce the following diagram; showing it commutes is equivalent to proving that that $\inp{h}=\inp{k}$.
\begin{equation}\label{eqn:algebracomp}
\xymatrixcolsep{3pc}
\xymatrix{
{} &\overline{S}\times\overline{\inp{Z}}\ar[d]^{\Delta\times\bigid } &{} \\
{} &\overline{S}\times\overline{S}\times\overline{\inp{Z}} \ar[d]^{\bigid \times \outp{f}\times\bigid } &{} \\
\overline{S}\times\overline{\outp{Y}}\times\overline{\inp{Z}} \ar[d]_{\bigid \times\overline{\inp{\psi}}} 
&\overline{S}\times\overline{\outp{X}}\times\overline{\inp{Z}} \ar[l]_-{\bigid \times\overline{\outp{\varphi}}\times\bigid } \ar[r]^-{\bigid \times\Delta\times\bigid }
&\overline{S}\times\overline{\outp{X}}\times\overline{\outp{X}}\times\overline{\inp{Z}} \ar[dd]^{\bigid \times\bigid \times\overline{\outp{\varphi}}\times\bigid } \\
\overline{S}\times\overline{\inp{Y}} \ar[d]_{\Delta\times\bigid } &{} &{} \\
\overline{S}\times\overline{S}\times\overline{\inp{Y}} \ar[r]_-{\bigid \times \outp{f}\times\bigid }
&\overline{S}\times\overline{\outp{X}}\times\overline{\inp{Y}} \ar[d]^{\bigid \times\overline{\inp{\varphi}}}
&\overline{S}\times\overline{\outp{X}}\times\overline{\outp{Y}}\times\overline{\inp{Z}} \ar[l]^-{\bigid \times\bigid \times\overline{\inp{\psi}}}\\
{} &\overline{S}\times\overline{\inp{X}}\ar[d]^{\inp{f}} &{} \\
{} &T\overline{S} &{}
}
\end{equation}

The commutativity of this diagram, which is dual to the one for associativity in (\ref{eqn:associativity}), holds in an arbitrary category with products. Although the middle square fails to commute by itself, the composite of the first two maps equalizes it; that is, the two composite morphisms \mbox{$\overline{S}\times\overline{\inp{Z}}\to\overline{S}\times \overline{\outp{X}}\times \overline{\inp{Y}}$} agree.

Since we proved the analogous result via string diagrams in the proof of Proposition \ref{prop:W is SMC}, we show it concretely using elements this time. Let $(s,z)\in\ol{S}\times\ol{\inp{Z}}$ be an arbitrary element. Composing six morphisms $\ol{S}\times\ol{\inp{Z}}\too\ol{S}\times\ol{\outp{X}}\times\ol{\inp{Y}}$ through the left of the diagram gives the same answer as composing through the right; namely,
$$\Big(s,\outp{f}(s),\inp{\psi}\big(\outp{\varphi}\circ\outp{f}(s),z\big)\Big)\in\ol{S}\times\ol{\outp{X}}\times\ol{\inp{Y}}.$$

Since the diagram commutes, we have shown that $\mcG$ is a functor. To prove that the pair $(\mcG,\mu)$ constitutes a lax monoidal functor $\bfW\to\mathbf{Set}$, i.e., a $\bfW$-algebra, we must establish coherence. Since $\mu$ simply consists of a coproduct and a product, this is straightforward and will be omitted.
\end{proof}

As established in Definition~\ref{def:SMC to Opd}, the coherence map $\mu$ allows us to define the operad algebra $\Opd{\mcG}$ from $\mcG$. This finally provides the formal setting to consider open dynamical systems over operadic wiring diagrams, such as our motivating one in Figure~\ref{fig:pipebrine}. We note that, in contrast to the trivial equality $\mcG(\Phi)S=S$ found in Definition~\ref{def:general algebra}, in the operadic setting we have \[\Opd{\mcG}(\Phi)(S_1,\ldots,S_n)=\amalg_{i=1}^n S_i.\] This simply means that the set of state variables of the larger box $Y$ is the disjoint union of the state variables of its constituent boxes $X_i$. Now that we have the tools to revisit Example~\ref{ex:promise}, we do so in the following section, but first we will define the subalgebra $\mcL$ to which it belongs---that of linear open systems.

\section{The Subalgebra of Linear Open Systems}
\label{sec:l}


In this section, we define the algebra $\mcL\taking\bfW_{\bfL}\to\Set$, which encodes linear open systems. Here $\bfW_{\bfL}$ is the category of $\bfL$-typed wiring diagrams, as in Remark \ref{rem:different Cs}. Of course, one can use Definition~\ref{def:SMC to Opd} to construct an operad algebra $\Opd{\mcL }:\Opd{\bfW_{\bfL}}\to\mathbf{Sets}$. 

Before we give a formal definition for $\mcL$, we first provide an alternative description for linear open systems and wiring diagrams in $\bfW_{\bfL}$. The category $\bfL$ enjoys special properties---in particular it is an additive category, as seen by the fact that there is an equivalence of categories $\bfL\cong \mathbf{Vect}_\RR$. Specifically, finite products and finite coproducts are isomorphic. Hence a morphism \mbox{$f:A_1\times A_2\to B_1\times B_2$} in  $\bfL$ canonically decomposes into a matrix equation \[\begin{bmatrix} a_1 \\ a_2 \end{bmatrix} \mapsto \begin{bmatrix} b_1 \\ b_2 \end{bmatrix} = \begin{bmatrix} f^{1,1} & f^{1,2} \\ f^{2,1} & f^{2,2} \end{bmatrix} \begin{bmatrix} a_1 \\ a_2 \end{bmatrix}\] This matrix is naturally equivalent to the whole map $f$ by universal properties. We use these to rewrite our relevant $\bfL$ maps in Definitions~\ref{def:rewrite} and \ref{def:rewrite2} below.

\begin{defn} \label{def:rewrite}
Suppose that $(M,\inp{U},\outp{U},f)$ is a linear open system and hence $f:M\times\inp{U}\to TM\times \outp{U}$. Then $f$ decomposes into the four linear maps:
\begin{align*}
f^{M,M}&\taking M\to TM & f^{M,U}&\taking \inp{U}\to TM \\ f^{U,M}&\taking M\to\outp{U} & f^{U,U}&\taking\inp{U}\to\outp{U}
\end{align*} By Definition~\ref{def:opensystem}, we know $f^{U,U}=0$. If we let $(m,\inp{u},\outp{u})\in M\times\inp{U}\times\outp{U}$, these equations can be organized into a single matrix equation
\begin{equation}\label{eqn:matrixform} \begin{bmatrix}\dot{m} \\ \outp{u} \end{bmatrix}=\begin{bmatrix} f^{M,M} & f^{M,U} \\ f^{U,M} & 0 \end{bmatrix}\begin{bmatrix} m \\ \inp{u} \end{bmatrix}
\end{equation}
\end{defn}

We will exploit this form in Definition~\ref{def:linear algebra} to define how $\mcL$ acts on wiring diagrams in terms of one single matrix equation, in place of the seemingly complicated commutative diagrams in (\ref{eqn:G of wd}). To do so, we also recast wiring diagrams in matrix format in Definition~\ref{def:rewrite2} below. 

\begin{defn} \label{def:rewrite2}
Suppose $\Phi=(X,Y,\varphi)$ is a wiring diagram in $\bfW_\bfL$. Recalling (\ref{eqn:prodwd}), we apply the dependent product functor to $\varphi$: 
\[\ol{\varphi}\taking\ol{\outp{X}}\times\ol{\inp{Y}}\to\ol{\inp{X}}\times\ol{\outp{Y}}\]
Since this is a morphism in $\bfL$, it can be decomposed into four linear maps
\begin{align*}
\overline{\varphi}^{X,X}&\taking \overline{\outp{X}}\to\overline{\inp{X}}&\overline{\varphi}^{X,Y}&\taking \overline{\outp{X}}\to\overline{\outp{Y}}\\
\overline{\varphi}^{Y,X}&\taking \overline{\inp{Y}}\to\overline{\outp{X}}&\overline{\varphi}^{Y,Y}&\taking \overline{\inp{Y}}\to\overline{\outp{Y}}
\end{align*}
By virtue of the no passing wires condition in Definition~\ref{def:W}, we must have \mbox{$\overline{\varphi}^{Y,Y}=0$}. We can then, as in (\ref{eqn:matrixform}), organize this information in one single matrix:

\[ \overline{\varphi}=
\begin{bmatrix} \;\overline{\varphi^{X,X}} &\overline{\varphi^{X,Y}}\; \\
\overline{\varphi^{Y,X}}
& 0 
\end{bmatrix}
\]
\end{defn}

\begin{rem}
The bijectivity condition in Definition~\ref{def:W} implies that $\overline{\varphi}$ is a permutation matrix.
\end{rem}

We now employ these matrix characterizations to define the algebra $\mcL$ of linear open systems.

\begin{defn} \label{def:linear algebra} We define the algebra $\mcL\taking(\bfW_\bfL,\oplus,0)\to (\Set,\times,\star)$ as follows. Let $X\in\Ob\bfW_{\bfL}$. Then the \emph{set of linear open systems $\mcL(X)$ on $X$} is defined as
\[\mcL(X):=\big\{(S,f)\;|\;S\in\Ob\TFS{\bfL}, (\overline{S},\overline{\inp{X}},\overline{\outp{X}},f)\in\Ob\ODS _\bfL\big\}.\]

Let $\Phi=(X,Y,\varphi)$ be a wiring diagram. Then, as in Definition~\ref{def:general algebra}, we define $\mcL (\Phi)(S,f):=(S,g)$. We use the format of Definitions~\ref{def:rewrite} and \ref{def:rewrite2} to define $g$:
\begin{equation} \label{eqn:glin}
\begin{split}  
g=
\begin{bmatrix} g^{S,S} & g^{S,X} \\ g^{X,S} & g^{X,X} \end{bmatrix} & =\begin{bmatrix} f^{S,X} & 0 \\ 0 & I \end{bmatrix}
\overline{\varphi}
\begin{bmatrix} f^{X,S} & 0 \\ 0 & I \end{bmatrix}+\begin{bmatrix} f^{S,S} & 0 \\ 0 & 0 \end{bmatrix} \\
& =\begin{bmatrix} f^{S,X} & 0 \\ 0 & I \end{bmatrix}
\begin{bmatrix}\ol{\varphi}^{X,X}&\ol{\varphi}^{X,Y}\\\ol{\varphi}^{Y,X}&\ol{\varphi}^{Y,Y}\end{bmatrix}
\begin{bmatrix} f^{X,S} & 0 \\ 0 & I \end{bmatrix}+\begin{bmatrix} f^{S,S} & 0 \\ 0 & 0 \end{bmatrix} \\
& =\begin{bmatrix} f^{S,X}\overline{\varphi}^{X,X}f^{X,S}+f^{S,S} & f^{S,X}\overline{\varphi}^{X,Y} \\ \overline{\varphi}^{Y,X}f^{X,S} & 0 \end{bmatrix}
\end{split}
\end{equation} 
This is really just a linear version of the commutative diagrams in (\ref{eqn:G of wd}). For example, the equation $g^{S,S}=f^{S,X}\overline{\varphi}^{X,X}f^{X,S}+f^{S,S}$ can be read off the diagram for $\inp{g}$ in (\ref{eqn:G of wd}), using the additivity of $\bfL$.

Finally, The coherence map $\mu_{\bfL_{X,X'}}:\mcL (X)\times\mcL (X')\to\mcL (X\oplus X')$ is given, as in Definition~\ref{def:mu}, by $\big((S,f),(S',f')\big)\mapsto (S+S',f\times f')$.
\end{defn}

We now establish that this constitutes an algebra.

\begin{prop} The pair $(\mcL,\mu_\bfL)$ of Definition~\ref{def:linear algebra} is a lax monoidal functor, i.e. a $\bfW_\bfL$-algebra.
\end{prop}

\begin{proof} Since coherence is identical to that in Proposition~\ref{prop:G is W-alg}, it will suffice to show functoriality. Let $\Phi=(X,Y,\varphi)$ and $\Psi=(Y,Z,\psi)$ be wiring diagrams with composition $\Psi\circ\Phi=(X,Z,\omega)$. We now rewrite $\overline{\omega}$ using a matrix equation in terms of $\overline{\varphi}$ and $\overline{\psi}$ by recasting (\ref{dia:composition diagrams}) in matrix form below.

\begin{equation} \label{eqn:omegamatrix}
\begin{split}
\overline{\omega}=\begin{bmatrix} \overline{\omega^{X,X}} & \overline{\omega^{X,Z}} \\ \overline{\omega^{Z,X}} & \overline{\omega^{Z,Z}} \end{bmatrix} & = \begin{bmatrix} \overline{\varphi}^{X,Y} & 0 \\ 0 & I \end{bmatrix}\overline{\psi}\begin{bmatrix} \overline{\varphi}^{Y,X} & 0 \\ 0 & I \end{bmatrix}+\begin{bmatrix} \overline{\varphi}^{X,X} & 0 \\ 0 & 0 \end{bmatrix} \\ 
& =\begin{bmatrix} \overline{\varphi}^{X,Y}\overline{\psi}^{Y,Y}\overline{\varphi}^{Y,X}+\overline{\varphi}^{X,X} & \overline{\varphi}^{X,Y}\overline{\psi}^{Y,Z} \\ \overline{\psi}^{Z,Y}\overline{\varphi}^{Y,X} & 0 \end{bmatrix}
\end{split}
\end{equation}

We now prove that $\mcL (\Psi\circ\Phi)=\mcL (\Psi)\circ\mcL (\Phi)$. We immediately have $\mcL (\Psi\circ\Phi)S=S=\mcL (\Psi)(\mcL (\Phi)S)$. Let $h:=\mcL (\Psi\circ\Phi)f$ and \mbox{$k:=\mcL (\Psi)(\mcL (\Phi)f)$}. We must show $h=k$. Let $g=\mcL (\Phi)f$ and $\Psi\circ\Phi=(X,Z,\omega)$. It is then straightforward matrix arithmetic to see that
\begin{equation} \label{eqn:matrix}
\begin{split}
k=\mcL (\Psi)g &=\begin{bmatrix} g^{S,Y} & 0 \\ 0 & I \end{bmatrix}\overline{\psi}\begin{bmatrix} g^{Y,S} & 0 \\ 0 & I \end{bmatrix} + \begin{bmatrix} g^{S,S} & 0 \\ 0 & 0 \end{bmatrix} \\ 
& =\begin{bmatrix}f^{S,X}(\overline{\varphi}^{X,Y}\overline{\psi}^{Y,Y}\overline{\varphi}^{Y,X}+\overline{\varphi}^{X,X})f^{X,S}+f^{S,S} & f^{S,X}\overline{\varphi}^{X,Y}\overline{\psi}^{Y,Z} \\
\overline{\psi}^{Z,Y}\overline{\varphi}^{Y,X}f^{X,S} & 0 \end{bmatrix} \\
&=\begin{bmatrix} f^{S,X} & 0 \\ 0 & I \end{bmatrix}\overline{\omega}\begin{bmatrix} f^{X,S} & 0 \\ 0 & I \end{bmatrix} + \begin{bmatrix} f^{S,S} & 0 \\ 0 & 0 \end{bmatrix} =\mcL (\Psi\circ\Phi)f=h
\end{split}
\end{equation}
Therefore, the pair $(\mcL,\mu_\bfL)$ constitutes a lax monoidal functor $\bfW_{\bfL}\to\mathbf{Set}$, i.e., a $\bfW_{\bfL}$-algebra.
\end{proof}

\begin{rem} Although we've been referring to $\mcL$ as a subalgebra of $\mcG$, this is technically not the case since they have different source categories. The following diagram illustrates precisely the relationship between the $\bfW_{\bfL}$-algebra $\mcL$, defined above, and the $\bfW$-algebra $\mcG$, defined in Section \ref{sec:g}.
\begin{equation} \label{eqn:final}
\xymatrix@C=16pt@R=30pt{
\bfW_{\bfL}\ar@{^{(}->}[rr]^{\bfW_i} \ar[dr]_{\mcL }&\ar@{}[d]|(.4){\overset{\textstyle\epsilon}{\Longrightarrow}}&\bfW\ar[dl]^{\mcG }\\
&\Set 
} 
\end{equation}
Here, the natural inclusion $\bfW_i\taking\bfW_{\bfL}\inj\bfW$ corresponds to $i\taking\bfL\hookrightarrow\Man$, and we have a natural transformation $\epsilon:\mcL \to\mcG \circ i$. Hence for each \mbox{$X\in\Ob\bfW_{\bfL}$}, we have a function $\epsilon_X:\mcL (X)\to \mcG (i(X))=\mcG (X)$ that sends the linear open system $(S,f)\in\mcL (X)$ to the open system \mbox{$(\TFS{i}(S),i(f))=(S,f)\in\mcG (X)$}.
\end{rem}

As promised, we now reformulate Example~\ref{ex:main} in terms of our language.

\begin{ex} \label{ex:as promised} For the reader's convenience, we reproduce Figure~\ref{fig:pipebrine} and Table~\ref{tab:explicit}.

\begin{figure}[ht]
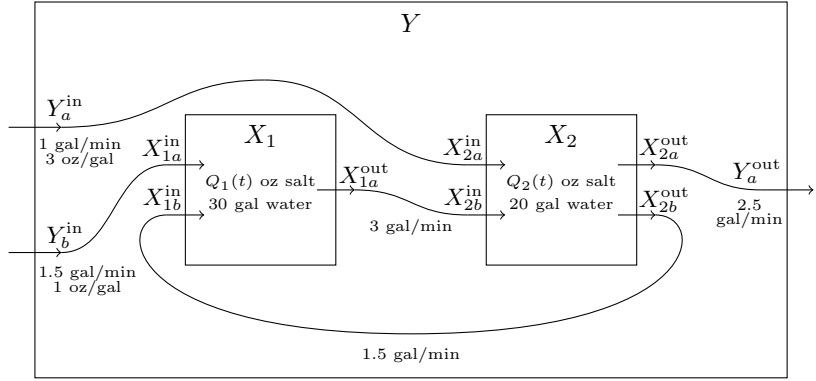

\activetikz{
	\path(0,0);
	\blackbox{(10,5)}{2}{1}{$Y$}{.7}
	        \node at (.4,3.6) {\small $\inp{Y}_{a}$};
	        \node at (.4,1.9) {\small $\inp{Y}_{b}$};
	        \node at (9.6,2.75) {\small $\outp{Y}_a$};
	\path(2,1.5);
	\blackbox{(2,2)}{2}{1}{$X_1$}{.5}
	        \node at (3,2.6) {\tiny $Q_1(t)$ oz salt};
	        \node at (3,2.3) {\tiny 30 gal water};
	        \node at (1.7,3.06) {\small $\inp{X}_{1a}$};
	        \node at (1.7,2.4) {\small $\inp{X}_{1b}$};
	        \node at (4.38,2.7) {\small $\outp{X}_{1a}$};
	\path(6,1.5);
	\blackbox{(2,2)}{2}{2}{$X_2$}{.5}
	        \node at (7,2.6) {\tiny $Q_2(t)$ oz salt};
	        \node at (7,2.3) {\tiny 20 gal water};
	        \node at (5.7,3.07) {\small $\inp{X}_{2a}$};
	        \node at (5.7,2.4) {\small $\inp{X}_{2b}$};
	        \node at (8.37,3.07) {\small $\outp{X}_{2a}$};
	        \node at (8.37,2.37) {\small $\outp{X}_{2b}$};
	\directarc{(4.25,2.5)}{(5.75,2.16667)} 
	    \node at (5,2) {\tiny 3 gal/min};
	\directarc{(0.35,1.6667)}{(1.75,2.83333)} 
	    \node at (.7,1.4) {\tiny 1.5 gal/min};
	    \node at (.7,1.2) {\tiny 1 oz/gal};
	\fancyarc{(0.35,3.3333)}{(5.75,2.83333)}{-40}{25} 
	    \node at (.6,3.1) {\tiny 1 gal/min};
	    \node at (.6,2.9) {\tiny 3 oz/gal};
	\directarc{(8.25,2.8333)}{(9.65,2.5)} 
	    \node at (9.5,2.3) {\tiny 2.5};
	    \node at (9.5,2.1) {\tiny gal/min};
	\fancyarc{(1.75,2.16667)}{(8.25,2.16667)}{20}{-45} 
	    \node at (5,.3) {\tiny 1.5 gal/min};
}
\caption{A dynamical system from Boyce and DiPrima interpreted over a wiring diagram $\Phi=(X_1,X_2;Y;\varphi)$ in $\Opd{\bfW}$.}
\end{figure}

\noindent\begin{minipage}{\linewidth}
\[
\begin{array}{c||c|c|c|c|c}
\rule[-4pt]{0pt}{16pt}
w\in\inp{X}+\outp{Y}&\inp{X}_{1a}&\inp{X}_{1b}&\inp{X}_{2a}&\inp{X}_{2b}&\outp{Y}_{a}
\\\hline
\rule[-4pt]{0pt}{16pt}
\varphi(w)\in\outp{X}+\inp{Y}&\inp{Y}_{b}&\outp{X}_{2b}&\inp{Y}_{a}&\outp{X}_{1a}&\outp{X}_{2a}
\end{array}
\]
\smallskip
\captionof{table}{}
\end{minipage}

We can invoke the yoga of Definition~\ref{def:rewrite2} to write $\overline{\varphi}$ as a matrix below:

\begin{equation} \label{eqn:phimatrix}
\begin{bmatrix} \;\overline{\outp{X_{1a}}}\; \\ \overline{\outp{X_{2a}}} \\ \overline{\outp{X_{2b}}} \\ \overline{\inp{Y_a}} \\ \overline{\inp{Y_b}} \end{bmatrix} =
\begin{bmatrix} 
0 &0 &I &0 &0 \\
0 &0 &0 &0 &I \\
0 &I &0 &0 &0 \\
I &0 &0 &0 &0 \\
0 &0 &0 &I &0
\end{bmatrix} \begin{bmatrix} \;\overline{\inp{X_{1a}}}\; \\ \overline{\inp{X_{1b}}} \\ \overline{\inp{X_{2a}}} \\ \overline{\inp{X_{2b}}} \\ \overline{\outp{Y_a}} \end{bmatrix}
\end{equation}

One can think of $\overline{\varphi}$ as a block permutation matrix consisting of identity and zero matrix blocks. An identity matrix in block entry $(i,j)$ represents the fact that the port whose state space corresponds to row $i$ and the one whose state space corresponds to column $j$ get linked by $\Phi$. In general, the dimension of each $I$ is equal to the dimension of the corresponding state space and hence the formula in~(\ref{eqn:phimatrix}) is true, independent of the typing. In the specific example of this system, however, all of these ports are typed in $\RR$, and so we have $I=1$ in~(\ref{eqn:phimatrix}).

As promised in Example~\ref{ex:promise}, we now write the open systems for the $X_i$ in Figure~\ref{fig:pipebrine} as elements of $\mcL (X_i)$. The linear open systems below in (\ref{eqn:tanks}) represent $f_1$ and $f_2$, respectively.
\begin{equation}
\label{eqn:tanks}
 \left[ \begin{array}{c} \dot{Q}_1 \\ \outp{X_{1a}} \end{array} \right] = \begin{bmatrix} -.1 & 1 & 1 \\ .1 & 0 & 0 \end{bmatrix} \left[ \begin{array}{c} Q_1 \\ \inp{X_{1a}} \\ \inp{X_{1b}} \end{array} \right], \left[ \begin{array}{c} \dot{Q}_2 \\ \outp{X_{2a}} \\ \outp{X_{2b}} \end{array} \right] = \begin{bmatrix} -.2 & 1 & 1 \\ .125 & 0 & 0 \\ .075  & 0 & 0\end{bmatrix} \left[ \begin{array}{c} Q_2 \\ \inp{X_{2a}} \\ \inp{X_{2b}} \end{array} \right]
\end{equation}

Note the proportion of zeros and ones in the $f$-matrices of (\ref{eqn:tanks})---this is perhaps why the making explicit of these details was an afterthought in (\ref{eqn:naive}). Because we may have arbitrary nonconstant coefficients, our formalism can capture more intricate systems. 

We then use (\ref{eqn:phimatrix}) to establish that $\inp{X}_{1b}=\outp{X}_{2b}$ and $\inp{X}_{2b}=\outp{X}_{1a}$. This allows us to recover the equations in (\ref{eqn:naive}):

\begin{displaymath}
   \left\{
     \begin{array}{lr}
       \dot{Q}_1=-.1Q_1+\inp{X_{1a}}+\inp{X_{1b}}=-.1Q_1+1.5+\outp{X_{2b}}=-.1Q_1+.075Q_2+1.5 \\
       \dot{Q}_2=-.2Q_2+\inp{X_{2a}}+\inp{X_{2b}}=-.2Q_2+3+\outp{X_{1a}}=-.2Q_2+.1Q_1+3
     \end{array}
   \right.
\end{displaymath}

The coherence map in Definition~\ref{def:linear algebra} gives us the combined tank system: \[(Q,f):=\mu_\bfL((\{Q_1\},f_1),(\{Q_2\},f_2))=(\{Q_1,Q_2\},f_1\times f_2)\in\mcL(X).\]  This system can then be written out as a matrix below
\begin{equation}\label{eqn:combinedsystem}\begin{bmatrix}\dot{Q_1} \\ \dot{Q_2} \\ \outp{X_{1a}} \\ \outp{X_{2a}} \\ \outp{X_{2b}}\end{bmatrix}=\begin{bmatrix} -.1 & 0 & 1 & 1 & 0 & 0 \\ 0 & -.2 & 0 & 0 & 1 & 1 \\ .1 & 0 & 0 & 0 & 0 & 0\\ 0 & .125 & 0 & 0 & 0 & 0 \\ 0 & .075 & 0 & 0 & 0 & 0  \end{bmatrix}\begin{bmatrix}Q_1 \\ Q_2 \\ \inp{X_{1a}} \\ \inp{X_{1b}} \\ \inp{X_{2a}} \\ \inp{X_{2ba}}\end{bmatrix}\end{equation} Finally, we can apply formula (\ref{eqn:glin}) to (\ref{eqn:combinedsystem}) above to express as a matrix the open system $(Q,g)=(\Phi)f\in\mcL(Y)$ for the outer box $Y$.

\[ \left[ \begin{array}{c} \dot{Q}_1 \\ \dot{Q}_2 \\ \outp{Y} \end{array} \right] = \begin{bmatrix} -.1 & .075 & 0 & 1 \\ .1 & -.2 & 1 & 0 \\ 0 & 1 & 0 & 0 \end{bmatrix} \left[ \begin{array}{c} Q_1 \\ Q_2 \\ \inp{Y_a} \\ \inp{Y_b} \end{array} \right] \]
\end{ex}

%
%

\medskip

\bibliographystyle{annotate}

\begin{thebibliography}{AMMO10}

\bibitem[Aba15]{Abadi}
Joseph Abadi.
\newblock Orthogonal factorization systems on 1- and 2-cob.
\newblock In preparation, 2015.

\bibitem[AMMO10]{Arthan}
Rob Arthan, Ursula Martin, Erik~A. Mathiesen, and Paulo Oliva.
\newblock A general framework for sound and complete {F}loyd-{H}oare logics.
\newblock {\em ACM Trans. Comput. Log.}, 11(1):Art. 7, 31, 2010.


\bibitem[Awo10]{Awodey}
Steve Awodey.
\newblock {\em Category {T}heory}, volume~52 of {\em Oxford Logic Guides}.
\newblock Oxford University Press, Oxford, second edition, 2010.


\bibitem[Bae13]{BaezTalk}
John Baez.
\newblock The foundations of applied mathematics.
\newblock ePrint online:\url{math.ucr.edu/home/baez/irvine/irvine.pdf}, 5 2013.

\bibitem[BB12]{Baez2}
John~C. Baez and Jacob Biamonte.
\newblock A course on quantum techniques for stochastic mechanics.
\newblock ePrint online: \url{www.arXiv.org/abs/1209.3632}, 2012.

\bibitem[BD65]{BD}
William~E. Boyce and Richard~C. DiPrima.
\newblock {\em Elementary {D}ifferential {E}quations and {B}oundary {V}alue
  {P}roblems}.
\newblock John Wiley \& Sons, Inc., New York-London-Sydney, 1965.


\bibitem[BS11]{Baez1}
J.~Baez and M.~Stay.
\newblock Physics, topology, logic and computation: a {R}osetta {S}tone.
\newblock In {\em New {S}tructures for {P}hysics}, volume 813 of {\em Lecture
  Notes in Phys.}, pages 95--172. Springer, Heidelberg, 2011.

\bibitem[Coe13]{Coecke}
Bob Coecke.
\newblock An alternative gospel of structure: order, composition, processes.
\newblock ePrint online: \url{www.arXiv.org/abs/1307.4038}, 2013.

\bibitem[DL10]{DL1}
Lee DeVille and Eugene Lerman.
\newblock Dynamics on networks {I}. {C}ombinatorial categories of modular
  continuous-time systems.
\newblock ePrint online: \url{www.arXiv.org/abs/1008.5359}, 2010.

\bibitem[DL14]{DL3}
Lee DeVille and Eugene Lerman.
\newblock Modular dynamical systems on networks.
\newblock {\em JEMS (to appear). \textnormal{ePrint online:
  \url{www.arXiv.org/abs/1303.3907}}}, 2014.


\bibitem[DL15]{DL2}
Lee DeVille and Eugene Lerman.
\newblock Dynamics on networks of manifolds.
\newblock {\em SIGMA Symmetry Integrability Geom. Methods Appl.}, 11:Paper 022,
  21, 2015.


\bibitem[Gro13]{Gromov}
Mikhail Gromov.
\newblock In a search for a structure, part 1: On entropy.
\newblock ePrint
  online:\url{www.ihes.fr/~gromov/PDF/structre-serch-entropy-july5-2012.pdf}, 6
  2013.

\bibitem[JS91]{JSTensor}
Andr{\'e} Joyal and Ross Street.
\newblock The geometry of tensor calculus. {I}.
\newblock {\em Adv. Math.}, 88(1):55--112, 1991.


\bibitem[JSV96]{JSTraced}
Andr{\'e} Joyal, Ross Street, and Dominic Verity.
\newblock Traced monoidal categories.
\newblock {\em Math. Proc. Cambridge Philos. Soc.}, 119(3):447--468, 1996.


\bibitem[KFA69]{KFA}
R.~E. Kalman, P.~L. Falb, and M.~A. Arbib.
\newblock {\em Topics in {M}athematical {S}ystem {T}heory}.
\newblock McGraw-Hill Book Co., New York-Toronto, Ont.-London, 1969.


\bibitem[Koe67]{Koestler}
Arthur Koestler.
\newblock {\em The {G}host in the {M}achine}.
\newblock Hutchinson \& Co., 1967.


\bibitem[Lei04]{Leinster}
Tom Leinster.
\newblock {\em Higher {O}perads, {H}igher {C}ategories}, volume 298 of {\em
  London Mathematical Society Lecture Note Series}.
\newblock Cambridge University Press, Cambridge, 2004.


\bibitem[ML98]{MacLane}
Saunders Mac~Lane.
\newblock {\em Categories for the {W}orking {M}athematician}, volume~5 of {\em
  Graduate Texts in Mathematics}.
\newblock Springer-Verlag, New York, second edition, 1998.


\bibitem[Sco71]{Scott}
Dana Scott.
\newblock The lattice of flow diagrams.
\newblock In {\em Symposium on {S}emantics of {A}lgorithmic {L}anguages}, pages
  311--366. Lecture Notes in Mathematics, Vol. 188. Springer, Berlin, 1971.

\bibitem[Sel11]{Selinger}
P.~Selinger.
\newblock A survey of graphical languages for monoidal categories.
\newblock In {\em New structures for physics}, volume 813 of {\em Lecture Notes
  in Phys.}, pages 289--355. Springer, Heidelberg, 2011.

\bibitem[Spi65]{SpiM-CalcMan}
Michael Spivak.
\newblock {\em Calculus on {M}anifolds. {A} {M}odern {A}pproach to {C}lassical
  {T}heorems of {A}dvanced {C}alculus}.
\newblock W. A. Benjamin, Inc., New York-Amsterdam, 1965.


\bibitem[Spi13]{Spivak2}
David~I. Spivak.
\newblock The operad of wiring diagrams: formalizing a graphical language for
  databases, recursion, and plug-and-play circuits.
\newblock ePrint online: \url{www.arXiv.org/abs/arXiv:1305.0297}, 2013.

\bibitem[Spi14]{Spivak}
David~I. Spivak.
\newblock {\em Category {T}heory for the {S}ciences}.
\newblock MIT Press, 2014.


\bibitem[SR13]{RupelSpivak}
David~I. Spivak and Dylan Rupel.
\newblock The operad of temporal wiring diagrams: formalizing a graphical
  language for discrete-time processes.
\newblock ePrint online: \url{www.arXiv.org/abs/1307.6894}, 2013.

\bibitem[SSR15]{SpivakSchultzRupel}
David~I. Spivak, Patrick Schultz, and Dylan Rupel.
\newblock Traced categories as lax functors out of free compact categories.
\newblock (In preparation), 2015.

\bibitem[Tes12]{Teschl}
Gerald Teschl.
\newblock {\em Ordinary {D}ifferential {E}quations and {D}ynamical {S}ystems},
  volume 140 of {\em Graduate Studies in Mathematics}.
\newblock American Mathematical Society, Providence, RI, 2012.


\bibitem[War83]{Warner}
Frank~W. Warner.
\newblock {\em Foundations of {D}ifferentiable {M}anifolds and {L}ie {G}roups},
  volume~94 of {\em Graduate Texts in Mathematics}.
\newblock Springer-Verlag, New York-Berlin, 1983.
\newblock Corrected reprint of the 1971 edition.


\end{thebibliography}


\end{document}